\documentclass[nohyperref]{article}

\usepackage{rotating,tabularx}
\usepackage{microtype}
\usepackage{graphicx}
\usepackage{subfigure}
\usepackage{bm}
\usepackage{booktabs} % for professional tables
\usepackage{multirow}

\usepackage{hyperref}
\usepackage[accepted]{icml2022}

\usepackage{amsmath}
\usepackage{amssymb}
\usepackage{mathtools}
\usepackage{amsthm}

\usepackage[capitalize,noabbrev]{cleveref}

\theoremstyle{plain}
\newtheorem{theorem}{Theorem}[section]
\newtheorem{prop}[theorem]{Proposition}
\newtheorem{lemma}[theorem]{Lemma}

\theoremstyle{definition}

\newtheorem{assumption}[theorem]{Assumption}
\theoremstyle{remark}

\usepackage[textsize=tiny]{todonotes}

\newenvironment{smallpmatrix}
    {\left(
    \begin{smallmatrix}}
    {\end{smallmatrix}
    \right)
    }
    \def\r{{\mathbf{r}}}
    
\newcommand{\R}{\mathbb{R}}

\newcommand{\x}{\mathbf x}
\newcommand{\y}{\mathbf y}
\newcommand{\bv}{\mathbf v}
\newcommand{\s}{\mathbf s}

\newcommand{\g}{\mathbf g}

\newcommand{\bb}{\mathbf b}

\newcommand{\z}{\mathbf z}

\newcommand{\w}{\mathbf w}

\newcommand{\q}{\mathbf q}
\newcommand{\argmin}{\mathop{\rm argmin}}
\newcommand{\br}{\mathbf r}
\newcommand{\bu}{\mathbf u}
\newcommand{\bz}{{\bf 0}}

\newcommand{\cO}{{\mathcal O}}
\newcommand{\tO}{{\tilde{\mathcal O}}}
\newcommand{\cK}{\mathcal K}
\newcommand{\cX}{\mathcal X}

\newcommand{\grad}{\mathrm{grad} }

\newcommand{\Hess}{\mathrm{Hess} }
\newcommand{\T}{\mathrm{T} }

\icmltitlerunning{Solving SPG-LS via Spherically
Constrained Least Squares Reformulation}

\begin{document}
\twocolumn[
\icmltitle{Solving  Stackelberg Prediction Game with Least
Squares Loss via Spherically Constrained Least Squares Reformulation}

\begin{icmlauthorlist}
\icmlauthor{Jiali Wang}{fdu}
\icmlauthor{Wen Huang}{xmu}
\icmlauthor{Rujun Jiang}{fdu}
\icmlauthor{Xudong Li}{fdu}
\icmlauthor{Alex L. Wang}{cmu}
\end{icmlauthorlist}
\icmlaffiliation{fdu}{School of Data Science, Fudan University, China}
\icmlaffiliation{xmu}{School of Mathematical Sciences, Xiamen University, China}
\icmlaffiliation{cmu}{School of Computer Science, Carnegie Mellon University, USA}

\icmlcorrespondingauthor{Rujun Jiang}{rjjiang@fudan.edu.cn}

\icmlkeywords{Machine Learning, ICML}

\vskip 0.3in
]
\printAffiliationsAndNotice{} % otherwise use the standard text.

\begin{abstract}
The Stackelberg prediction game (SPG) is popular in characterizing strategic interactions between a learner and an attacker. As an important special case, the SPG with least squares loss (SPG-LS) has recently received much research attention. Although initially formulated as a difficult bi-level optimization problem, SPG-LS admits tractable reformulations which can be polynomially globally solved by semidefinite programming or second order cone programming. However, all the available approaches are not well-suited for handling large-scale datasets, especially those with huge numbers of features.  In this paper, we explore an alternative reformulation of the SPG-LS. By a novel nonlinear change of variables, we rewrite the SPG-LS  as a spherically constrained least squares (SCLS) problem. Theoretically, we show that an $\epsilon$ optimal solution to the SCLS (and the SPG-LS) can be achieved in $\tO(N/\sqrt{\epsilon})$ floating-point operations, where $N$ is the number of nonzero entries in the data matrix. Practically, we apply two well-known methods for solving this new reformulation, i.e., the Krylov subspace method and the Riemannian trust region method. Both algorithms are factorization free so that they are suitable for solving large scale problems. Numerical results on both synthetic and real-world datasets indicate that the SPG-LS, equipped with the SCLS reformulation, can be solved orders of magnitude faster than the state of the art.

\end{abstract}
\section{Introduction}
The big data era has led to an explosion in the availability of data from which to make decisions. It is thus indispensable to use machine learning techniques to gain deep insights from massive data. In practice, many classic data analytic approaches start by splitting available data into the training and test sets. Then, learning algorithms are fed with the training set and are expected to produce results which generalize well to the test set. However, this paradigm only works under the key implicit assumption that the available data in both training and test sets are independently and identically distributed, which, unfortunately, is not always the truth in practice. For example, in the context of email spam filtering, an attacker often adversarially generates spam emails based on his knowledge of the spam filter implemented by the email service provider \cite{bruckner2011stackelberg, zhou2019survey}. In addition to malicious attacks, sometimes the data providers may manipulate data for their own interests. For instance, health insurance policy holders may decide to modify self-reported data to reduce their premiums. On the other hand, the insurers (the ``defenders'' in this scenario) aim to select a good price model for the true data despite only seeing the modified data.

In fact, these scenarios can be modeled by the Stackelberg prediction game (SPG) \cite{bruckner2011stackelberg, shokri2012protecting, zhou2016modeling, wahab2016stackelberg, zhou2019survey,bishop2020optimal} which characterizes the interactions between two players, a learner (or, a leader) and a data provider (or, a follower). {In this setting, the learner makes the first move by selecting a learning model. Then the data provider, with full knowledge of the learner’s model, is allowed to modify its data. The learner's goal is to
minimize its own loss function under the assumption that the training data has been optimally modified from the data provider’s perspective.} From the above description, we see that the SPG model concerns two levels of optimization problems: The follower optimally manipulates its data and the leader makes its optimal decision taking into account the data manipulation. Formally, it is often formulated as a hierarchical mathematical problem or a bi-level optimization problem, which is generally NP-hard even in the  simplest case with linear constraints
and objectives~\cite{jeroslow1985polynomial}.

To overcome this issue, \citet{bishop2020optimal} take the first step to focus on a subclass of SPGs that can be reformulated as fractional programs. Specifically, they assume that all the loss functions of the leader and the follower are least squares, and that a quadratic regularizer is added to the follower's loss to penalise its manipulation of the data.
This assumption eventually turns the bi-level optimization problem into a single-level fractional optimization task which is proven to be polynomially globally solvable. Since no other assumption is made about the learner and data provider, this subclass of SPG, termed as the SPG-LS, is general enough to be applied in wide fields. However, the bisection algorithm proposed in \citet{bishop2020optimal} involves solving several tens of semidefinite programs (SDPs) which are computationally prohibitive in practice. Later, \citet{pmlr-v139-wang21d} improves over \citet{bishop2020optimal} by showing that the SPG-LS can be globally solved via solving only a single SDP with almost the same size as the ones in \citet{bishop2020optimal}. Furthermore, this single SDP can be reduced to a second order cone program (SOCP). It is shown in \citet{pmlr-v139-wang21d} that the SOCP approach for solving SPG-LS can be over 20,000+ times faster than the bisection method proposed in \citet{bishop2020optimal}. Yet, the SOCP method is still not well-suited for solving large-scale SPG-LS. Indeed, the spectral decomposition in the SOCP reformulation process is
time-consuming when the future dimension is high. This inevitably reduces the practical applicability of the SOCP approach for the SPG-LS.

% main cost of the method is a spectral decomposition for the data matrix, which is time-consuming and might obtain inaccurate solution especially the related matrix is large and dense. This inexactness may finally cause unstable algorithm performances which makes the applications in real life impossible.

In this paper, we present a novel reformulation to resolve the above mentioned issues of the SOCP method. Specifically, a nonlinear change of variables is proposed to reformulate the SPG-LS as a spherically constrained least squares (SCLS) problem. Then, we prove that an optimal solution to the SPG-LS can be recovered easily from any optimal solution to the SCLS under a mild assumption.
The SCLS can be seen as an equality constrained version of the trust region subproblem \cite{conn2000trust}, which admits a large amount of existing research on practical algorithms and theoretical complexity analysis.
Based on this, we show that an $\epsilon$ optimal solution\footnote{We say $\bar\x$ is an $\epsilon$ optimal solution for an optimization problem $\min_{\x\in \cX} f(\x)$, if $\bar\x\in\cX$ and $f(\bar\x)\le \min_{\x\in \cX} f(\x)+\epsilon$.} of the SCLS and thus SPG-LS can be solved in $\tO(N/\sqrt{\epsilon})$ flops, where $N$ denotes the number of nonzero entries of the data matrix and $\tO(\cdot)$ hides the logarithmic factors. This means there exits a linear time algorithm for finding an $\epsilon$ optimal solution of the SPG-LS.
%Moreover, we show that the obtained SCLS can be solved globally in polynomial time.
%It thus implies that the spectral decomposition step in the SOCP reformulation process, which has high computational costs and often lead to inaccurate results, can be avoid.
%We further provide two efficient algorithms for solving the proposed GTRS reformulation.
Moreover, we demonstrate the empirical efficiency of our SCLS reformulation when matrix factorization free methods like the Krylov subspace method  \cite{gould1999solving,zhang2018nested} and the Riemannian trust region Newton (RTRNewton) method \cite{absil2007trust} are used as solvers.

%Moreover, our GTRS reformulation can take better advantage of data matrix and convert the multiplication of two matrices which takes $\cO(n^3)$ operations, into two matrix-vector multiplication which takes $\cO(n^2)$ operations. Actually, to efficiently solve GTRS, we exhaust three possible cases which involves solving classical trust region subproblem (TRS) or reformulated inequality constrained GTRS.

%\lxd{Rewrite this part}
%In general, according to \citet{adachi2017solving}, algorithms for TRS has following three categories: approximation methods~\cite{gould1999solving,gould2010solving,steihaug1983conjugate,toint1981towards}, accurate methods for the dense TRS~\cite{adachi2017solving,more1983computing,nocedal2006numerical} and accurate methods for the large-sparse TRS~\cite{golub1991quadratically,hager2001minimizing,gould2010solving,rendl1997semidefinite,rojas2001new}. Meanwhile, there are numerous methods solving GTRS under various assumptions (see~\cite{more1993generalizations,stern1995indefinite,sturm2003cones,ben1996hidden,feng2012duality,pong2014generalized} and the references therein) .

% We perform numerical experiments to valid our new reformulation. In our experiment, for TRS, a nested Lanczos Method \cite{zhang2018nested} is used to obtain a feasible approximation in the Krylov subspace.

We summarise our contributions as follows:
\begin{itemize}
    \item We derive an SCLS reformulation for the SPG-LS that avoids spectral decomposition steps (which are expensive when the involved data matrices are large). Moreover, we show that an optimal solution to the SPG-LS can be recovered from any optimal solution to the SCLS reformulation under a mild condition.

    \item Based on the reformulation, we show that an $\epsilon$ optimal solution for the SCLS can be found using $\tO(1/\sqrt{\epsilon})$ matrix vector products. In other words, an $\epsilon$ solution can be obtained in running time $\tO(N/\sqrt{\epsilon})$, where $N$ is the number of nonzeros in the data matrix. Moreover, we show that an $\epsilon$ optimal solution of SCLS can be used to recover an $\epsilon$ optimal solution for the original SPG-LS.

    \item Two practically efficient algorithms, which are factorization free, are adopted to solve the SCLS reformulation. We show that the SCLS approach significantly outperforms the
    SOCP approach with experiments on both real and synthetic data sets.
\end{itemize}
\section{Preliminaries}
In this section, we elaborate on the SPG-LS problem adopting the same terminology as in \citet{bishop2020optimal,pmlr-v139-wang21d}. To have a better understanding of our reformulation, a brief review of  methods in \citet{pmlr-v139-wang21d}, which is the  fastest existing method for solving the SPG-LS, will also be provided.

We assume that the learner has access to $m$ sample tuples $S = \{(\mathbf{x}_i, y_i, z_i)\}_{i=1}^m$, where $\mathbf{x}_i\in \mathbb{R}^n$ is input data with $n$ features, $y_i$ and ${z_i}$ are the true output label of $\mathbf{x}_i$ and the label that the data provider would like to achieve, respectively.
These samples are assumed to follow some fixed but unknown distribution $\mathcal{D}$.
The learner aims at training a linear predictor $\w \in \mathbb{R}^n$ to best estimate the true output label $y_i$ given the fake data.
Meanwhile, the data provider, with full knowledge of the learner’s predictive model $\w$, selects its own strategy (i.e., the modified data $\mathbf{\hat x}_i$) to make the corresponding prediction $\w^T\mathbf{\hat{x}}_i$ close to the desired label $z_i$.  Note that there is also a regularizer, $\gamma>0$, to control the deviation from the original data $\x_i$. This hyper-parameter adjusts the trade-off between data manipulation and closeness to the aimed target.

The problem can be modeled as a Stackelberg prediction game \cite{bruckner2011stackelberg,bishop2020optimal}. On the one hand, each data provider aims to minimize its own loss function with a regularizer that penalizes the manipulation of the data by solving the following optimization problem:
$$
\mathbf{x}_{i}^{*}=\underset{\hat{\mathbf{x}}_{\mathbf{i}}}{\operatorname{argmin}}\left\|\mathbf{w}^{T} \hat{\mathbf{x}}_{i}-z_{i}\right\|^{2}+\gamma\left\|\mathbf{x}_{i}-\hat{\mathbf{x}}_{i}\right\|_{2}^{2} \quad i \in[m],
$$
where $\w$ is the learner's model parameter that is known to the data provider.
On the other hand, the learner seeks to minimize the least squares loss with the modified data $\{\x_i^*\}_{i=1}^m$:
$$
\mathbf{w}^{*} \in \underset{\mathbf{w}}{\operatorname{argmin}}\sum_{i=1}^m\left \|\mathbf{w}^{T} \mathbf{x}_{i}^{*}-y_{i}\right\|^{2},
$$
To find the Stackelberg equilibrium of the two players, we focus on the following bi-level optimization problem
\begin{equation}
\label{pb:matrixform}\small
\begin{array}{ccl}
& \underset{\mathbf{w}}{\min}\ &\left\|X^{*} \mathrm{w}- \y\right\|^{2} \\
&\ \text{s.t.}\ &X^{*}=\underset{\hat{X}}{\argmin}\  \|\hat{X}{\mathbf{w}}-\mathbf{z}\|^2 +\gamma\|\hat{X}-X\|_{F}^{2},
\end{array}
\end{equation}
where the $i$-th row of ${X} \in\R^{m\times n}$ is the input sample $\x_i$ and the $i$-th entries of $\y, \z\in \R^m$ are labels $y_i$ and $z_i$, respectively.

In the following section, we have a quick review of single SDP and SOCP methods in \citet{pmlr-v139-wang21d}.
\subsection{SDP Reformulation}
By using the Sherman-Morrison formula \cite{sherman1950adjustment}, the SPG-LS can be rewritten as a quadratic fractional program~\cite{bishop2020optimal}
\begin{equation}
\label{pb:ori}
\inf_\w~~ \left\|\frac{\frac{1}{\gamma}\z\w^T\w+X\w}{1+\frac{1}{\gamma}\w^T\w}-\y\right\|^2.
\end{equation}
Introducing an augmented variable $\alpha=\w^T\w/\gamma$, we have the following quadratic fractional  programming (QFP) reformulation.
\begin{equation}
\label{pb:ori_with_alpha}
\begin{array}{lll}
\inf_{\w,\,\alpha}& v(\w,\alpha) \triangleq \left\|\frac{\alpha\z+X\w}{1+\alpha}-\y\right\|^2 \, \\
{\rm s.t.}&{\w}^T\w=\gamma\alpha.
\end{array}
\end{equation}

\begin{lemma}[Theorem 3.3 in \citet{pmlr-v139-wang21d}]
        Problem \eqref{pb:ori_with_alpha} is equivalent to the following SDP
        \begin{align}
        \label{pb:SDP}
        \begin{array}{lll}
        &\sup_{\mu,\lambda}& \mu\\
        &\rm s.t.&A-\mu B+\lambda C\succeq0,
        \end{array}
        \end{align}
where
$A=  \begin{smallpmatrix}
        X^T X& X^{T}(\mathbf{z}-\mathbf{y}) & -X^T \mathbf{y}\\
        (\mathbf{z}-\mathbf{y})^{T}X&\|\mathbf{z}-\mathbf{y}\|^2&-(\mathbf{z}-\mathbf{y})^T \mathbf{y}\\
        -\mathbf{y}^T X & -\mathbf{y}^T (\mathbf{z}-\mathbf{y})&\mathbf{y}^T \mathbf{y}\\
        \end{smallpmatrix},
\displaystyle B=
          \begin{smallpmatrix}
          \cO_{n}& & \\
           &1&1\\
           &1&1\\
          \end{smallpmatrix}\text{ and }
C =\begin{smallpmatrix}
          \frac{{I}_n}{\gamma}& &\\
          & 0 & -\frac{1}{2}\\
          & -\frac{1}{2}&0\\
          \end{smallpmatrix}.$
%$$
%\begin{array}{l}
%\label{eq:formABC}
%A=  \begin{smallpmatrix}
%        X^T X& X^{T}(\mathbf{z}-\mathbf{y}) & -X^T \mathbf{y}\\
%        (\mathbf{z}-\mathbf{y})^{T}X&\|\mathbf{z}-\mathbf{y}\|^2&-(\mathbf{z}-\mathbf{y})^T \mathbf{y}\\
%        -\mathbf{y}^T X & -\mathbf{y}^T (\mathbf{z}-\mathbf{y})&\mathbf{y}^T \mathbf{y}\\
%        \end{smallpmatrix},\\
%\displaystyle B=
%          \begin{smallpmatrix}
%          O_{n}& & \\
%           &1&1\\
%           &1&1\\
%          \end{smallpmatrix}\text{ and }
%C =\begin{smallpmatrix}
%          \frac{{I}_n}{\gamma}& &\\
%          & 0 & -\frac{1}{2}\\
%          & -\frac{1}{2}&0\\
%          \end{smallpmatrix}.
%\end{array}$$
Here $\cO_n$ denotes a $n\times n$ matrix with all entries being zeros and $I_n$ denotes the $n\times n$ identity matrix.
\end{lemma}
\subsection{SOCP Reformulation}
\citet{pmlr-v139-wang21d} further constructed an invertible matrix $V$ such that $A$, $B$ and $C$ are simultaneously congruent to arrow matrices via the change of variables associated to $V$, i.e.,
\begin{equation*}
\label{eq:formtilA}
\tilde A:= V^TAV =
\begin{pmatrix}
D & \bb\\
\bb^T&c\\
\end{pmatrix},
\end{equation*}
where $D = {\rm Diag}(d_1,\ldots,d_{n+1})\in\R^{(n+1)\times (n+1)}$, $\bb\in\R^{n+1}$ and $c\in\R$, and
\begin{equation*}
\label{eq:formtilBC}\tilde B:= VBV =\begin{smallpmatrix}
\cO_{n+1}&\\
&4
\end{smallpmatrix}{~ and ~} \tilde C= V^T CV= \begin{smallpmatrix}
\frac{1}{\gamma}I_{n+1}&\\
&-1
\end{smallpmatrix}.
\end{equation*}
% with $V$ being some given invertible matrix.
%\lxd{What is $D$ and $d_i$ in Lemma 2.2?}
Therefore the linear matrix inequality (LMI) constraint in \eqref{pb:SDP} is equivalent to
%\begin{equation*}
%\label{eq:tildeLMI}
$\tilde A-\mu\tilde B+\lambda \tilde C\succeq 0.$
%\end{equation*}
Using the generalized Schur complement, %we obtain that the above LMI is further equivalent to
%    \begin{equation*}
%    \label{eq:Schur}
%    \begin{array}{l}
%    D+\tfrac{\lambda}{\gamma}I_{n+1} \succeq 0,\\
%    \bb\in {\rm Range}(D+\tfrac{\lambda}{\gamma}I_{n+1}),\\
%    c-4\mu-\lambda-\bb^{T}(D+\tfrac{\lambda}{\gamma}I_{n+1} )^{\dagger}\bb \succeq  0.
%    \end{array}
%    \end{equation*}
we further obtain an SOCP reformulation as follows.
\begin{lemma}[Theorem 4.1 in \citet{pmlr-v139-wang21d}]
With the same notation in this section, problem \eqref{pb:SDP} is  equivalent to the following SOCP problem
\begin{align}
\label{pb:SOCP}
    \begin{array}{ll}
         \sup_{\mu,\lambda,\s}& \mu\\
         \rm s.t.&d_i+\tfrac{\lambda}{\gamma}\ge0, ~ i\in [n+1], \\  &c-4\mu-\lambda-\sum_{i=1}^{n+1}s_i\ge0,\\   &s_i(d_i+\tfrac{\lambda}{\gamma})\ge~ b_i^2,~s_i\ge0,~i \in [n+1].
    \end{array}
\end{align}
\end{lemma}
To obtain the SOCP reformulation, we need a spectral decomposition to a matrix of order $(n+1)\times (n+1)$ \cite{pmlr-v139-wang21d}, which is expensive when the dimension is high and may lead to inaccurate solutions when the matrix is ill-conditioned.
To amend this issue, we obtain in the next section a factorization free method based on a novel reformulation  of problem \eqref{pb:ori}.
\section{Main Results}
\label{sec:3}
In this section, we show that using a nonlinear change of variables, we can rewrite \eqref{pb:ori_with_alpha} as a least squares problem over the unit sphere. This is the key observation of our paper.

Before presenting our main results, we first make a blanket assumption on the nonemptiness of the optimal solution set of \eqref{pb:ori}.
\begin{assumption}\label{asmp:bound}
Assume that the optimal solution set of  \eqref{pb:ori} (or equivalently, \eqref{pb:ori_with_alpha}) is nonempty.
\end{assumption}

%The above assumption is reasonable since if the parameter is unbounded, then the learner will find that either there is not enough sample points or the data was manipulated by the adversary. In general, the learner can add a regularizer to avoid this phenomenon, for example, adding $\lambda \|\w\|^2$ to his loss function as people do in ridge regression.
%Then with similar derivations we have the following model,
%\begin{equation*}
%\begin{array}{lll}
%&\inf_{\w,\,\alpha}&~~ \left\|\frac{\alpha\z+X\w}{1+\alpha}-\y\right\|^2+\lambda \|\w\|^2\\
%& {\rm s.t.}&~~{\w}^T\w=\gamma\alpha,
%\end{array}
%\end{equation*}
%which has an additional $\lambda\|\w\|^2$  in the objective function compared to \eqref{pb:ori_with_alpha}.
%Note that the above objective function is coercive in $\w$. This, together with $\alpha=\w^T\w/\gamma$, prevents the optimal solution being unbounded.

Our main result is that under Assumption \ref{asmp:bound}, the QFP \eqref{pb:ori_with_alpha} can be reformulated as a spherical constrained least squares (SCLS) problem
\begin{equation}
\label{pb:scls}
\begin{array}{lll}
&\min\limits_{\tilde\w,\,\tilde\alpha}& {} {\tilde v}(\tilde \w,\tilde \alpha) \triangleq \left\|\frac{\tilde \alpha}{2} \z + \frac{\sqrt{\gamma}}{2}X\tilde \w - (\y - \frac{\z}{2})\right\|^2\\
&{\rm s.t.}& {} \tilde\w^T\tilde\w + \tilde \alpha^2 = 1.
\end{array}
\end{equation}
A formal statement is deferred to Theorem \ref{thm:main}.

Before presenting our main results, we first introduce two lemmas that show how, given a feasible solution in \eqref{pb:ori_with_alpha}, we can construct a feasible solution with the same objective value in \eqref{pb:scls} and vice versa (up to a minor achievability issue).

\begin{lemma}
\label{lem:tilde_w_tilde_alpha_construction}
Suppose $(\w,\alpha)$ is a feasible solution of \eqref{pb:ori_with_alpha}. Then $(\tilde\w,\tilde\alpha)$, defined as
\begin{equation}
\label{eq:tilde_w_tilde_alpha_construction}
\tilde \w \coloneqq \frac{2}{\sqrt{\gamma}(\alpha + 1)}\w
\quad\text{and}\quad
\tilde \alpha \coloneqq \frac{\alpha - 1}{\alpha + 1},
\end{equation}
is feasible to \eqref{pb:scls} and $v(\w,\alpha) = \tilde v(\tilde \w, \tilde \alpha)$.
%with the same objective value as $(\w,\alpha)$ in \eqref{pb:ori_with_alpha}.
\end{lemma}
\begin{proof}
We first note that $(\tilde \w,\tilde \alpha)$ are well-defined as $\alpha\geq 0$ by the feasibility of $(\w,\alpha)$ in \eqref{pb:ori_with_alpha}.
Next, we check feasibility of $(\tilde\w,\tilde\alpha)$ in \eqref{pb:scls}:
\begin{align*}
\tilde \w^T \tilde \w + \tilde \alpha^2 &= \frac{4}{\gamma(\alpha + 1)^2}\w^T\w + \frac{(\alpha - 1)^2}{(\alpha + 1)^2}\\
&=\frac{4\alpha + (\alpha - 1)^2}{(\alpha + 1)^2}~= 1.
\end{align*}
Here, the second equality follows from the fact that $\w^T\w = \gamma\alpha$.
Finally, we see that%{\small
\begin{align*}
  &\tilde v(\tilde\w,\tilde\alpha) \\
= &\left\|\frac{\tilde \alpha}{2} \z + \frac{\sqrt{\gamma}}{2}X\tilde \w - (\y - \frac{\z}{2})\right\|^2\\ =& \left\|\frac{\alpha -1}{2(\alpha + 1)} \z + \frac{\sqrt{\gamma}}{2}X\frac{2}{\sqrt{\gamma}(\alpha + 1)}\w - (\y - \frac{\z}{2})\right\|^2\\
=& \left\|\frac{\alpha\z + X\w}{\alpha + 1} - \y\right\|^2 = v(\w,\alpha).
\end{align*}%}
This completes the proof.
\end{proof}

\begin{lemma}
\label{lem:w_alpha_construction_neq_one}
Suppose $(\tilde \w, \tilde\alpha)$ is feasible to \eqref{pb:scls} with $\tilde \alpha \neq 1$. Then $(\w,\alpha)$, defined as
\begin{equation}
\label{eq:w_alpha_construction_neq_one}
\w \coloneqq \frac{\sqrt{\gamma}}{1-\tilde\alpha}\tilde\w
\quad\text{and}\quad
\alpha \coloneqq \frac{1+\tilde \alpha}{1-\tilde\alpha},
\end{equation}
is feasible to \eqref{pb:ori_with_alpha}  and $\tilde v (\tilde \w, \tilde \alpha) = v(\w ,\alpha)$.
\end{lemma}
\begin{proof}
We first note that $(\w,\alpha)$ are well-defined as $\tilde\alpha\neq 1$.
Next, we check feasibility of $(\w,\alpha)$ in \eqref{pb:ori}:
\begin{align*}
\w^T\w &= \frac{\gamma}{(1-\tilde\alpha)^2}\tilde\w^T\tilde\w = \frac{\gamma(1-\tilde\alpha^2)}{(1-\tilde\alpha)^2} = \frac{\gamma(1+\tilde\alpha)}{1-\tilde \alpha} =\gamma\alpha.
\end{align*}
Here, the second equality follows from the fact that $\tilde\w^T\tilde\w + \tilde\alpha^2 = 1$.
Finally, we check the objective value of $(\w,\alpha)$:
$\begin{array}{ll}
& \left\|\frac{\alpha\z + X\w}{1+\alpha}-\y\right\|^2 = \left\|\frac{(1+\tilde \alpha)\z + \sqrt{\gamma}X\tilde\w}{(1-\tilde \alpha)+(1+\tilde\alpha)}-\y\right\|^2 \\
=& \left\|\frac{\tilde \alpha}{2}\z + \frac{\sqrt{\gamma}}{2}X\tilde\w -(\y-\z/2)\right\|^2.
\end{array}$
\end{proof}
Let $v^*$ and $\tilde v^*$ be the optimal values of  \eqref{pb:ori_with_alpha} and \eqref{pb:scls}, respectively. Now we are ready to present our main results.
\begin{theorem}
\label{thm:main}
Given Assumption \ref{asmp:bound}, then there exists an optimal solution $(\tilde \w, \tilde\alpha)$ to \eqref{pb:scls} with $\tilde \alpha \neq 1$. Moreover, $(\w,\alpha)$, defined by \eqref{eq:w_alpha_construction_neq_one},
is an optimal solution to \eqref{pb:ori_with_alpha} and $v^* = v(\w,\alpha) = \tilde v(\tilde \w, \tilde \alpha) = \tilde v^*$.
\end{theorem}
\begin{proof}
Since the feasible region of  \eqref{pb:scls} is compact and $\tilde v$ is continuous, it follows from the well-known Weierstrass theorem that there exists at least one optimal solution to \eqref{pb:scls}.

%
%Provided Assumption \ref{asmp:bound},  the set $\{(\w,\alpha):\w^T\w=\alpha,~\alpha\le M\}$ is compact and the objective function in \eqref{pb:ori_with_alpha} is continuous in $(\w,\alpha)$ over this set. Therefore, there exists an optimal solution for \eqref{pb:ori_with_alpha}. \lxd{Assumption 3.1 assumes the nonemptyness. It seems to me that we only needs to assume the nonemptyness not the boundedness.}
%Similarly, as the feasible region of  \eqref{pb:scls} is compact, there exists an optimal solution for \eqref{pb:scls}.

Note that if $(\tilde \w, \tilde \alpha)$ with $\tilde \alpha=1$ is a  feasible solution  in \eqref{pb:scls}, we must have $\tilde \w=\bz$.
Now we claim that $(\bz, 1)$ cannot be the \emph{unique} optimal solution to \eqref{pb:scls}. Suppose on the contrary that $(\bz,1)$ is the unique optimal solution to \eqref{pb:scls}.
%Note that
%\begin{equation}
%\label{eq:f201}
%\tilde v(\bz,1)=\|\z-\y\|^2.
%\end{equation}
Let $(\w^*,\alpha^*)$ be any optimal solution to \eqref{pb:ori_with_alpha}.
Then, from Lemma \ref{lem:tilde_w_tilde_alpha_construction}, we have
\begin{equation}
\label{eq:f12}
v(\w^*,\alpha^*)= \tilde v (\tilde\w^*,\tilde\alpha^*),
\end{equation}
where $
\tilde \w^* \coloneqq \frac{2}{\sqrt{\gamma}(\alpha^* + 1)}\w^*,~
%\quad\text{and}\quad
\tilde \alpha^* \coloneqq \frac{\alpha^* - 1}{\alpha^* + 1} < 1,
$
 and $(\tilde\w^*,\tilde\alpha^*)$ is feasible to \eqref{pb:scls}.
%Note that \eqref{eq:tilde_w_tilde_alpha_construction} implies $\tilde\alpha\neq 1$.
Since  $(\bz,1)$ is the unique optimal solution to \eqref{pb:scls}, it holds that
\begin{equation}
\label{eq:f2til}
\tilde v(\tilde\w^*,\tilde\alpha^*)>\tilde v(\bz,1) = \|\z-\y\|^2.
\end{equation}
On the other hand, for any $\w\neq 0$ and $t>0$, the pair $(t\w,t^2\w^T\w/\gamma)$  is clearly feasible to \eqref{pb:ori_with_alpha} with objective value
\begin{equation}
\label{eq:f2limit}
\begin{array}{ll}
&v(t\w,t^2\w^T\w/\gamma)\\
=&\left\|\frac{t^2 \w^T\w/\gamma}{1 + t^2 \w^T\w/\gamma}\z + \frac{t}{1 + t^2 \w^T\w/\gamma}X\w - \y\right\|^2\\
\to& \|\z-\y\|^2,\text{~as~}t\to \infty.
\end{array}
\end{equation}
%Indeed, as $\w$ is nonzero, the first summand approaches $\z$ as $t\to\infty$ and the second summand approaches $\bz$ as $t\to\infty$.
Consequently, for sufficiently large $t$, we must have from \eqref{eq:f12}, \eqref{eq:f2til} and \eqref{eq:f2limit} that
$$v(\w^*,\alpha^*)>v(t\w,t^2\w^T\w/\gamma),$$
which contradicts the optimality of $(\w^*,\alpha^*)$ to \eqref{pb:ori_with_alpha}.

The above claim shows that there exists an optimal solution $(\tilde \w, \tilde\alpha)$ to \eqref{pb:scls} with $\tilde \alpha \neq 1$. Then, Lemma \ref{lem:w_alpha_construction_neq_one} yields $\tilde v(\tilde \w, \tilde \alpha) = \tilde v^* = v(\w, \alpha) \ge v^*$ with $(\w, \alpha)$ defined in \eqref{eq:w_alpha_construction_neq_one}.
Similarly, under Assumption \ref{asmp:bound},
from Lemma \ref{lem:tilde_w_tilde_alpha_construction}, we see that $v^*\ge  \tilde v^*$. Thus, $v^* = \tilde v^*$. The proof is completed.
\end{proof}
We remark that the other direction of above theorem also holds. That is, under Assumption \ref{asmp:bound}, there exists an optimal solution $(\w,\alpha)$ to \eqref{pb:ori_with_alpha}, and, furthermore, $(\tilde\w,\tilde\alpha)$, defined by \eqref{eq:tilde_w_tilde_alpha_construction}, is optimal to \eqref{pb:scls}.
We also remark that using the relationship \eqref{eq:w_alpha_construction_neq_one} and the equivalence of \eqref{pb:ori} and \eqref{pb:ori_with_alpha}, an $\epsilon$ optimal solution of the SCLS can be used to
recover an $\epsilon$ optimal solution of the SPG-LS.
%with the same objective value as $(\w,\alpha)$ in \eqref{pb:ori_with_alpha} .

Letting
$\hat{L} = \begin{pmatrix}\frac{\sqrt{\gamma}}{2}X&\frac{\z}{2}\end{pmatrix}\text{~and~} \br=\begin{smallpmatrix}
\tilde{\mathbf{w}}\\ \tilde{\alpha}
\end{smallpmatrix},$
we can rewrite problem \eqref{pb:scls} in a more compact form
\begin{equation}
\label{pb:rscls}
\min_{\br}~q(\br) \quad  {\rm s.t.}~\br^T\br = 1,
\end{equation}
where $q(\br)$ is a least squares
\begin{equation}
\label{eq:quad}
    q(\br) = \|\hat{L}\br-(\y-\z/2)\|_2^2 = \br^TH\br + 2\g^T\br + p
\end{equation}
with
$H  = \hat{L}^T\hat{L},~\g = \hat{L}^T(\z/2-\y),~ p = (\z/2-\y)^T(\z/2-\y).$

In the following of this paper, we focus on solving \eqref{pb:rscls}.
Problem \eqref{pb:rscls} is closed related to the well-known trust region subproblem (TRS), where the sphere constraint is replaced by a unit ball constraint.
There exist various methods for solving \eqref{pb:rscls} from the literature on TRS \cite{more1983computing,gould1999solving, conn2000trust,hazan2016linear,ho2017second,zhang2017generalized,zhang2018error,zhang2018nested,carmon2018analysis}, or the
generalized trust region subproblem (GTRS)  \cite{more1993generalizations,pong2014generalized,jiang2019novel,jiang2020linear,wang2020generalized,wang2021implicit}, which minimizes a (possible nonconvex) quadratic function over a (possible nonconvex)  quadratic inequality or equality constraint. Note that the TRS differs from the SCLS in the constraint, and the GTRS contains the SCLS as a special case.

\section{Complexity and Algorithms}
In this section, we first show that in theory there exists a linear time algorithm to find an $\epsilon$ optimal solution for the SPG-LS. After that, we introduce two practically efficient algorithms to solve \eqref{pb:rscls} (and thus recover a solution for the SPG-LS).
%\subsection{Theoretical compleixty}

We point out that the linear time algorithms for the TRS \cite{hazan2016linear,ho2017second} can be adapted to design a linear time algorithm with complexity $\tO(N/\sqrt{\epsilon})$ for the SCLS to achieve an $\epsilon$ optimal solution, and the linear time algorithms for the GTRS \cite{jiang2020linear,wang2020generalized,wang2021implicit} indicate that the SCLS, as a special case of the GTRS, can also be solved in linear time $\tO(N/\sqrt{\epsilon})$.
Here $N$ denotes the number of nonzero entries in the data matrix, and the logarithm in the runtime comes from the probability of success in Lanczos type methods for finding the smallest eigenvalue of a matrix. Once we obtain a solution $\tilde\r$ such that
$\|\tilde\r\|=1$ and $q(\tilde\r) \le q(\r^*) + \epsilon$, where $\r^*$ is an optimal solution of \eqref{pb:rscls}, we can set $\begin{pmatrix}
\tilde{\mathbf{w}}\\ \tilde{\alpha}
\end{pmatrix}=\tilde\br$ and
$(\w,\alpha)$ as in $\eqref{eq:w_alpha_construction_neq_one}$.
Then $\w$ is an $\epsilon$ optimal solution to \eqref{pb:ori} because $v(\w,\alpha)=\tilde v(\tilde\w,\tilde\alpha)=q(\r)$ and $q(\r^*)=v^*$ for the same reasoning .
Thus, one can obtain an $\epsilon$ optimal solution to SPG-LS in runtime
$\tO\left(N/\sqrt{\epsilon}\right)$ as the main cost is in solving the SCLS \eqref{pb:rscls}.

However, in practice the computation of even an approximate minimum eigenvalue may be expensive. Instead, we will introduce two highly efficient algorithms to solve \eqref{pb:scls} without computing
approximate eigenvalues.
One is the Krylov subspace method (adapted to the spherically constrained case) proposed in \citet{gould1999solving},
and the other is the Riemannian trust region Newton (RTRNewton) method proposed in \citet{absil2007trust}.
%In this subsection, we first describe the Krylov subspace method (adapted to the spherical constrained case) with complexity analysis and then give a description for the RTRNewton method.
\subsection{The Krylov Subspace Method}
\label{sec:ksm}
The simplest idea of the Krylov subspace method (Section 5 in \citet{gould1999solving}) solves a sequence of smaller dimensional problems in the same form of  \eqref{pb:rscls}. Specifically, define  $(k+1)$st Krylov subspace
\[
\cK_k:=\{\g,H\g,H^2\g,\ldots,H^k\g\}.
\]
Let $Q_k=[\q_0,\q_1,\ldots,\q_k]\in\R^{(n+1)\times (k+1)}$ be an orthonormal basis produced by the generalized Lanczos process.
Then assuming $\dim \cK_k=k+1$, we have that $Q_k^THQ_k$ is a tridiagonal matrix with $Q_k^TQ_k=I_{k+1}$.
Each iteration of the Krylov subspace method solves the following subproblem (adapted to the spherical constrained case)
\begin{equation}
\label{pb:subsp}
\min_{r\in \cK_k,\|\br\|= 1} \br^TH\br + 2\g^T\br + p.
\end{equation}
\citet{gould1999solving} proved that the above subproblem can be solved efficiently in $\cO(k)$ flops, if we use a safeguarded Newton's method, where the most expensive cost is $k$ matrix-vector products for $H^t\g$, with $t\in[k]$. We remark that though \citet{gould1999solving} considered the case  $\|\br\|\le 1$, the two cases $\|\br\|\le 1$ and $\|\br\|=1$ are essentially the same if the inequality in the TRS is active, which occur if $\|(H-\lambda_{\min}I)^\dagger\g\|>1$\footnote{This can be easily derived from the proof of Proposition \ref{prop:expk}.}.
Here $(\cdot)^\dagger$ denotes the Moore-Penrose pseudoinverse of a matrix $(\cdot)$.
%{\cbl Specifically, the necessary and sufficient optimality conditions for the trust region subproblem (TRS) is $\left(H+\lambda^{*} I\right) \mathbf{r}^{*}=-\mathbf{g}$, $H+\lambda^{*} I \succeq 0, \lambda^{*} \geq 0,\left\|\mathrm{r}^{*}\right\| \leq 1$ and $\lambda^{*}\left(\left\|\mathrm{r}^{*}\right\|^{2}-1\right)=0$. If the inequality in the TRS is active, then we have $\left\|\mathrm{r}^{*}\right\|-1= 0$ and $\lambda^{*} \ge \lambda_{\min}$ where $\lambda_{\max}$ and $\lambda_{\min}$ be the largest and smallest eigenvalues for $H$. Assume, without loss of generality, $\lambda_{\max}>\lambda_{\min}$. . To
%obtain $\lambda^{*} \geq 0$ in above optimality conditions of TRS, one can add $-\lambda_{\min }\left(\|\mathbf{r}\|^{2}-1\right)$ to the objective function of the SCLS. Accordingly, its optimality condition can be written as $(H + \lambda^* I-\lambda_{\min } I)\br^* = -\g, ~ H + \lambda^* I-\lambda_{\min } I\succeq0, ~ \|\br^*\| = 1$.  By $H+\lambda^{*} I-\lambda_{\min } I \succeq 0$ we have $\lambda^{*} \geq 0$ holds. This shows that the two cases $\|\mathbf{r}\| \leq 1$ and $\|\mathbf{r}\|=1$ are essentially the same.}
%Indeed, the problem \eqref{pb:subsp} can be reduced to a problem in form \eqref{pb:rscls} with a tridiagonal Hessian, which can be solved by a safeguarded Newton's method in cost $\cO(k)$.

To achieve better practical performance, \citet{gould1999solving} proposed the generalized Lanczos trust-region (GLTR) method,  which is an efficient implementation of the above Krylov subspace method. Based on an efficient nested restarting strategy, \citet{zhang2018nested} further proposed a nested Lanczos method for TRS (LTRSR), which is an improvement for GLTR.

The convergence behavior of the Krylov subspace method is also well analyzed in the literature.
%Before introducing the convergence analysis for the Krylov subspace method, let us first recall the easy case and hard case for \eqref{pb:scls}, which is adapted for its variants TRS \cite{conn2000trust} and GTRS \cite{pong2014generalized}.
%\lxd{Definitions of hard and easy cases here?}
The optimality condition of problem \eqref{pb:rscls}  is characterized as follows (adapted from Chapter 7 in \citet{conn2000trust})
\begin{equation}
\label{eq:opt}
\begin{array}{rl}
(H + \lambda^* I)\br^* &= -\g,\\
H + \lambda^* I&\succeq0,\\
\quad \|\br^*\| &= 1,
\end{array}
\end{equation}
where $\lambda^*$ is the corresponding Lagrangian multiplier.
It is shown that there always exists an optimal solution $\br^*$ and a unique Lagrangian multiplier $\lambda^*$, because different $\lambda^*$s yield different values for $\|\r^*\|$, contradicting $\|\r^*\|=1$. %\lxd{Citation here?}

%If $\lambda^*>-\lambda_{\min}$, then $\x^*=(H + \lambda^* I)^{-1}\g$ is unique.
Define
\begin{equation}
\label{eq:kappa}
\kappa=\frac{\lambda_{\max}+\lambda^*}{\lambda_{\min}+\lambda^*},
\end{equation}
for $\lambda^*\ge-\lambda_{\min}$ and use the convention $\frac{1}{0}=\infty$. Here $\kappa$ is regarded as the condition number of \eqref{pb:rscls} \cite{carmon2018analysis}.
\citet{zhang2018nested} and \citet{carmon2018analysis} demonstrated that when $\kappa<\infty$, the Krylov subspace method satisfies
\[
f(\br_k)-f(\br^*)\le \cO(\exp(-k/\sqrt\kappa)),
\]
where $\br_k$ is an optimal solution of \eqref{pb:subsp}, and $\br^*$ is an optimal solution for \eqref{pb:rscls}.
\citet{carmon2018analysis} further proved that for all cases including $\lambda^*=-\lambda_{\min}$, a variant of the Krylov subspace method  where $\g$ is perturbed with random vector will output a solution $\r_k$ satisfies
\[
f(\br_k)-f(\br^*)\le \tO(1/k^2).
\]
for the SCLS (adapted from their analysis for the TRS).
This indeed gives another $\tO(N/\sqrt\epsilon)$ time algorithm for solving the SPG-LS up to $\epsilon$ tolerance; see also \citet{wang2021implicit} for extensions of this idea to the GTRS.

Next we relate the existing convergence results with problem \eqref{pb:rscls}.
\begin{prop}
\label{prop:expk}
If $\|(H-\lambda_{\min} I)^{\dagger}\g\|>1$, we must have $\lambda^*>-\lambda_{\min}$ and thus
\[
f(\br_k)-f(\br^*)\le \cO(\exp(-k/\sqrt\kappa)),
\]
for $\kappa$ defined in \eqref{eq:kappa}.
\end{prop}
\begin{proof}
Note that if $\lambda^*=-\lambda_{\min}$, then the first equation in \eqref{eq:opt} implies that
$\r^*=-(H + \lambda^* I)^\dagger\g$ and thus the assumption in the proposition implies
$\|\r^*\|>1$. However, this violates the constraint $\|\r^*\|=1$.
Therefore we must have $\lambda^*>-\lambda_{\min}$ and thus $\kappa<\infty$.
\end{proof}
In fact, we checked the data in the experiments in Section \ref{sec:5} and found that $\|\bar\br\|>1$ always holds in real-world datasets and our synthetic datasets. %, which means \eqref{pb:rscls} is always in the easy case.

\subsection{The Riemannian Trust Region Newton (RTRNewton) Method}

The feasible set in Problem~\eqref{pb:rscls} forms a unit sphere $\mathcal{S}^{n} = \{\br \in \mathbb{R}^{n + 1} : \br^T \br = 1\}$. When $\mathcal{S}^{n}$ is endowed with the Euclidean metric $\langle \bv, \bu \rangle = \bv^T \bu$, the unit sphere is a Riemannian manifold~\cite{AMS2008}. Therefore, the RTRNewton method proposed in~\citet{absil2007trust} can be used. The RTRNewton method for Problem~\eqref{pb:scls} is summarized in Algorithm~\ref{alg:RTRNewton}.

\begin{algorithm}[ht!]
\caption{A Riemannian Trust Region Newton Method}
\label{alg:RTRNewton}
\begin{algorithmic}[1]
\REQUIRE Initial iterate $\br_0$, real numbers
$\overline{\Delta} > 0$, $\Delta_0 \in (0, \overline{\Delta})$, %$\nu \in (0, 1)$,
$c \in (0, 0.25)$, $\tau_1 \in (0, 1)$, and $\tau_2 > 1$;
%\ENSURE An estimation of the minimizer $x_*$;
\FOR{$k=0,1,2,\ldots$}
\STATE Obtain $s_k \in \mathbb{R}^d$ by (approximately) solving
\begin{align}
\label{eq:subrtr}
\s_k \approx \argmin_{\|\s\|_2 \leq \Delta_k} m_k(\s),
\end{align}
where $m_k(\s) = q(\br_k) + \s^T \grad \, q(\br_k) + \frac{1}{2} \s^T \Hess \; q(\br_k) [\s]$, $\grad \; q$ denotes the Riemannian gradient of~$q$ in~\eqref{Rgrad}, and $\Hess \; q$ denotes the Riemannian Hessian of~$q$ in~\eqref{RHess}; \label{RTRNewton:st1}
\STATE Set
$\rho_k \gets \frac{q(\br_k) - q(R_{\br_k}(\s_k))}{m_k(0) - m_k(\s_k)}$, where $R$ denotes the retraction in~\eqref{Retr};
\IF {$\rho_k > c$}
\STATE $\br_{k + 1} \gets R_{\br_k}(\s_k)$;
\ELSE
\STATE $\br_{k + 1} \gets \br_k$;
\ENDIF
\IF {$\rho_k > \frac{3}{4}$}
\STATE If $\|\s_k\| \geq 0.8 \Delta_k$, then $\Delta_{k + 1} \gets \min(\tau_2 \Delta_k, \overline{\Delta})$; otherwise $\Delta_{k+1} \gets \Delta_k$;
%\If {$\|\eta_k\| \geq 0.8 \Delta_k$}
%\State $\Delta_{k + 1} \gets \tau_2 \Delta_k$;
%\Else
%\State $\Delta_{k+1} \gets \Delta_k$;
%\EndIf
\ELSIF {$\rho_k < 0.1$}
\STATE $\Delta_{k + 1} \gets \tau_1 \Delta_k$;
\ELSE
\STATE $\Delta_{k + 1} \gets \Delta_k$;
\ENDIF
\ENDFOR
\end{algorithmic}
\end{algorithm}

Algorithm~\ref{alg:RTRNewton} relies on the notion of Riemannian gradient, Riemannian Hessian, and retraction. We refer to~\citet{AMS2008} for their rigorous definitions.
Here, we give the Riemannian gradient, the Riemannian Hessian for Problem~\eqref{pb:rscls} and the used retraction.

% use the same techniques in Section~6.4 of~\cite{AMS2008} to derive the required optimization tools below.
The Riemannian gradient of~$q$ is given by
\begin{equation} \label{Rgrad}
\grad \; q(\br) = P_{\T_{\br} \mathcal{S}^n} \nabla q(\br) = (I - \br \br^T) (2 H \br + 2 \g),
\end{equation}
where $P_{\T_{\br} \mathcal{S}^n}$ denotes the orthogonal projection onto the tangent space at $\br$ with $\T_{\br} \mathcal{S}^n = \{\s : \s^T \br = 0 \}$, and $\nabla q(\br)$ denotes the Euclidean gradient of~$q$, i.e., $\nabla q(\br) = 2 H \br + 2 \g$. The action of the Riemannian Hessian of~$q$ at $\br$ along $\bv \in\T_{\br} \mathcal{S}^n$ is given by
\begin{align}
\Hess \; q(\br) [\bv] =& P_{\T_{\br} \mathcal{S}^n} ( \nabla^2 q(\br) \bv - \bv \br^T \nabla q(\br) ), \nonumber \\
=& (I - \br \br^T) (2 H \bv - 2 \bv \br^T (H \br + \g)), \label{RHess}
\end{align}
where $\nabla^2 q(\br)=2H$ denotes the Euclidean Hessian of $q$ at $\br$. The retraction $R$ that we use is given by
\begin{equation} \label{Retr}
R_{\br}(\bv) = \frac{ \br + \bv }{\|\br + \bv\|},
\end{equation}
where $\br \in \mathcal{S}^n$ and $\bv \in \T_{\br} \mathcal{S}^n$.

The subproblem \eqref{eq:subrtr} is approximately solved by the truncated conjugate gradient method. We use the implementations in ROPTLIB~\cite{huang2018roptlib}.

% The global convergence of RTRNewton has been established in Theorem~7.4.2 of~\citet{AMS2008}. Moreover, one can verify that, for Problem~\eqref{pb:rscls}, the all assumptions therein hold by Proposition~7.4.5 and Corollary~7.4.6  of~\citet{AMS2008}.
% The local superlinear convergence rate is also guaranteed by that Theorem~7.4.11 of~\citet{AMS2008}, that Retraction~\eqref{Retr} is second order, and that the function~$q$ is $C^{\infty}$.

The global convergence and local superlinear convergence rate of RTRNewton
have been established by Theorem~7.4.4 and Theorem~7.4.11 of~\citet{AMS2008}.
We state the results in the theorem below.
\begin{theorem}\label{thm:RTRN}
Let $\{\br_k\}$ be a sequence of iterates generated by Algorithm~1. It follows that
\[
\lim_{k \rightarrow \infty} \grad\; q(\br_k) = 0.
\]
Suppose $\br^*$ is a nondegenerate local minimizer of $q$, i.e., $\grad\; q(\br^*) = 0$ and
$\Hess\; q(\br^*)$ is positive definite. Then there exists $c > 0$ such that for all sequence
$\{\br_k\}$ generated by Algorithm~1 converging to $\br^*$, there exists $K>0$ such that
for all $k > K$,
\begin{equation} \label{rtrquadratic}
\mathrm{dist}(\br_{k+1}, \br^*) \leq c~\mathrm{dist}(\br_k, \br^*)^2.
\end{equation}
\end{theorem}
The proof for the theorem is deferred to \cref{ap:RTRN}.
The global convergence rate has also been established where the iteration complexity is $\mathcal{O}\left(\epsilon_{g}^{-2}\right)$ for $\|\operatorname{grad}\; q(x)\| \leq \epsilon_{g}$. We refer interested readers to Theorem~3.9 of~\citet{boumal2019global}.

% [[BAC2018: Global rates of convergence for nonconvex optimization on manifolds, IMANUM, 2018]]

\subsection{Time complexity comparisons}
In this section, we give a theoretical worst case time complexity of different methods for solving the SPG-LS. First we point out that the RTRNewton cannot be guaranteed to converge to the \emph{global} minimum of SPG-LS. In the worst case, the RTRNewton needs to solve $\mathcal{O}\left(\epsilon_{g}^{-2}\right)$ many trust region subproblems. This means the time complexity is much worse than the Krylov subspace method \cite{carmon2018analysis} studied in Section \ref{sec:ksm}.

Next we compare the time complexity of the Krylov subspace method and the SOCP method.
In the case of dealing with a sparse data matrix, the time complexity of the Krylov subspace method is $\tO\left(N/\sqrt{\epsilon}\right)$, where $N$ is the number of nonzero entries in the data matrix $X$. Here, we use the fact that the cost of the matrix-vector product in the Krylov subspace method is $\cO(N)$ as we can compute $H\r$ by $\hat L^T(\hat L\r)$ for any given $\r \in \mathbb{R}^n$.
If $\kappa<\infty$, then the complexity can be further improved to $\cO(N\log(1/\epsilon))$.
In the dense case with $m = {\cal O}(n)$, the complexity is $\tO\left(N/\sqrt{\epsilon}\right)$ and can be improved to ${\cal O}(n^2 \log(1/\epsilon)$ if $\kappa < \infty$.
Next we consider the time complexity for the SOCP method, which consists of the time complexity of formulating the matrix $A$, the spectral decomposition and  the IPM  for solving the SOCP. Since the spectral decomposition and the IPM can not benefit much from the data sparsity, we do not distinguish the sparse and dense cases for the SOCP method. Particularly, the cost of  formulating the matrix $A$ is lower bounded by $\cO(N)$ and upper bounded by $\cO(n^2)$ and the spectral decomposition takes $O\left(n^{\omega}\right)$ flops for some $\omega$ satisfying $2<\omega<3$ \cite{demmel2007fast}. Meanwhile, the iteration complexity for solving the $\mathrm{SOCP}$ reformulation is $\cO(\sqrt{n} \log (1 / \epsilon))$ according to \citet{monteiro2000polynomial}. As per iteration in cost in the IPM is $\cO(n)$, the total cost of the IPM is $\cO\left(n^{\frac{3}{2}} \log (1 / \epsilon)\right)$. Therefore the worst case complexity of the SOCP method is $\cO\left(n^{w}+n^{\frac{3}{2}} \log (1 / \epsilon)\right).$
% Let the time complexity of formulating the matrix $A$ be $\cO(D)$, which is at least $\cO(N)$.
% The iteration complexity for solving the $\mathrm{SOCP}$ reformulation is $\cO(\sqrt{n} \log (1 / \epsilon))$ according to \citet{monteiro2000polynomial}. As per iteration in cost in the IPM is $\cO(n)$, the total cost of the IPM is $\cO\left(n^{\frac{3}{2}} \log (1 / \epsilon)\right)$ for SOCP. Besides, the spectral decomposition takes $O\left(n^{\omega}\right)$ flops for some $\omega$ satisfying $2<\omega<3$ \cite{demmel2007fast}. Therefore the worst case complexity of the SOCP method is \[\cO\left(D+n^{w}+n^{\frac{3}{2}} \log (1 / \epsilon)\right).\]
% To compare the complexity between the Krylov subspace method and the SOCP method, we consider the case that the data matrix is dense and $m=\cO(n)$.
% If $\kappa<\infty$, the complexity of the Krylov subspace method is $\cO(n^2\log(1/\epsilon))$, which is certainly much better than that of the SOCP method.

Theortically, it is hard to compare the Krylov subspace method and the SOCP method as the result depends on $\kappa$, $N$, $w$ and $\epsilon$.
% In general,  the Krylov subspace method costs $\tO(n^2/\sqrt\epsilon)$, and its comparison with the SOCP method depends on $n^{\omega-2}$ and
% $1 / \sqrt{\epsilon}$.
% It is difficult to say which one outperforms the other.
% One should also note that the complexity for LTRSR is the worst case complexity. Moreover,
In practice, it usually holds that $\kappa<\infty$ and the spectral decomposition step in the SOCP methods usually costs $\cO(n^3)$. In fact, our experiments show that the spectral decomposition step often needs more time than the IPM for the SOCP. Thus, the Krylov subspace method, which can effectively utilize the data sparsity, is much faster than the SOCP approach especially for solving large-scale problems.

% and the SOCP is much slower than the SCLS methods for large dimensional instances.
% The above comparison further indicates that the Krylov subspace method has a lower time complexity than the SOCP method in the sparse case, as the SOCP method cannot utilize the sparsity of data matrix. Indeed, the (full) spectral decomposition, which is the main cost in the SOCP method, cannot utilize the sparsity of the data matrix.
%Thus we do not have a time complexity for SOCP method in the sparse setting.
%Finally, we give a remark on the dense case. In the dense case, the matrix vector product $H\r$ can be either computed $H\r$ directly with first forming $H$, or computed by $\hat L^T(\hat L\r)$ if $m<n/2$.

\section{Experiment Results}
\label{sec:5}
In this section, we present numerical results on both synthetic and real-world datasets to verify the superiority of our proposed reformulation in terms of computational costs.
We refer a nested Lanczos Method for TRS (LTRSR)~\cite{zhang2018nested} to perform the GLTR method, and use the implementation of Riemannian trust-region Newton (RTRNewton)~\cite{absil2007trust} from  Riemannian Manifold Optimization Library (ROPTLIB)~\cite{huang2018roptlib}. Similar as the setting in \citet{pmlr-v139-wang21d}, we compare the running time of above two methods with the SDP and SOCP approaches in \citet{pmlr-v139-wang21d}, averaged over 10 trials, to evaluate the performance of our new reformulation.
All the four methods solve the SDP, SOCP or the SCLS reformulations to their default precision and the solutions to the SPG-LS are recovered accordingly.
%then recover solutions for the SPG-LS from the reformulations.
%We point out that all the tested algorithms achieve high accuracy. Indeed, the objective values and final MSES on test data reported by four different algorithms have relative gaps of order $10^{-8}$ and $10^{-7}$, respectively.
%Meanwhile, the final MSEs for  on the test data are also with relative errors of order $10^{-7}$.
%Thus we omit comparisons on objective values and MSEs in our tables and only report the wall-clock time among these approaches in this section.

All simulations are implemented using MATLAB R2021b on a PC running Windows 10 Intel(R) Xeon(R) E5-2650 v4 CPU (2.2GHz) and 64GB RAM.
We report the results of  two real datasets and six synthetic datasets and defer other results to the supplementary material.  In all following tests, the parameter $\gamma$ is set as 0.1. %More results on

\subsection{Real-world Dataset}

We first compare four methods on the red wine dataset \cite{cortez2009modeling}, which consist of 1599 instances each with 11 features.
The output label is a physiochemical measurement
ranged from 0 to 10, where a higher score means that the corresponding wine sample has better quality.
Wine producers would like to manipulate the data to fool the learner to predict a higher score when the raw label is smaller than some threshold $t$.
We consider the case that there are two kinds of providers $\mathcal{A}_{\text{modest}}$ and $\mathcal{A}_{\text{severe}}$, where the details of the manipulation are set the same as \citet{pmlr-v139-wang21d}.
\begin{figure}[!htb]
\centering
\subfigure{
\includegraphics[scale=0.27]{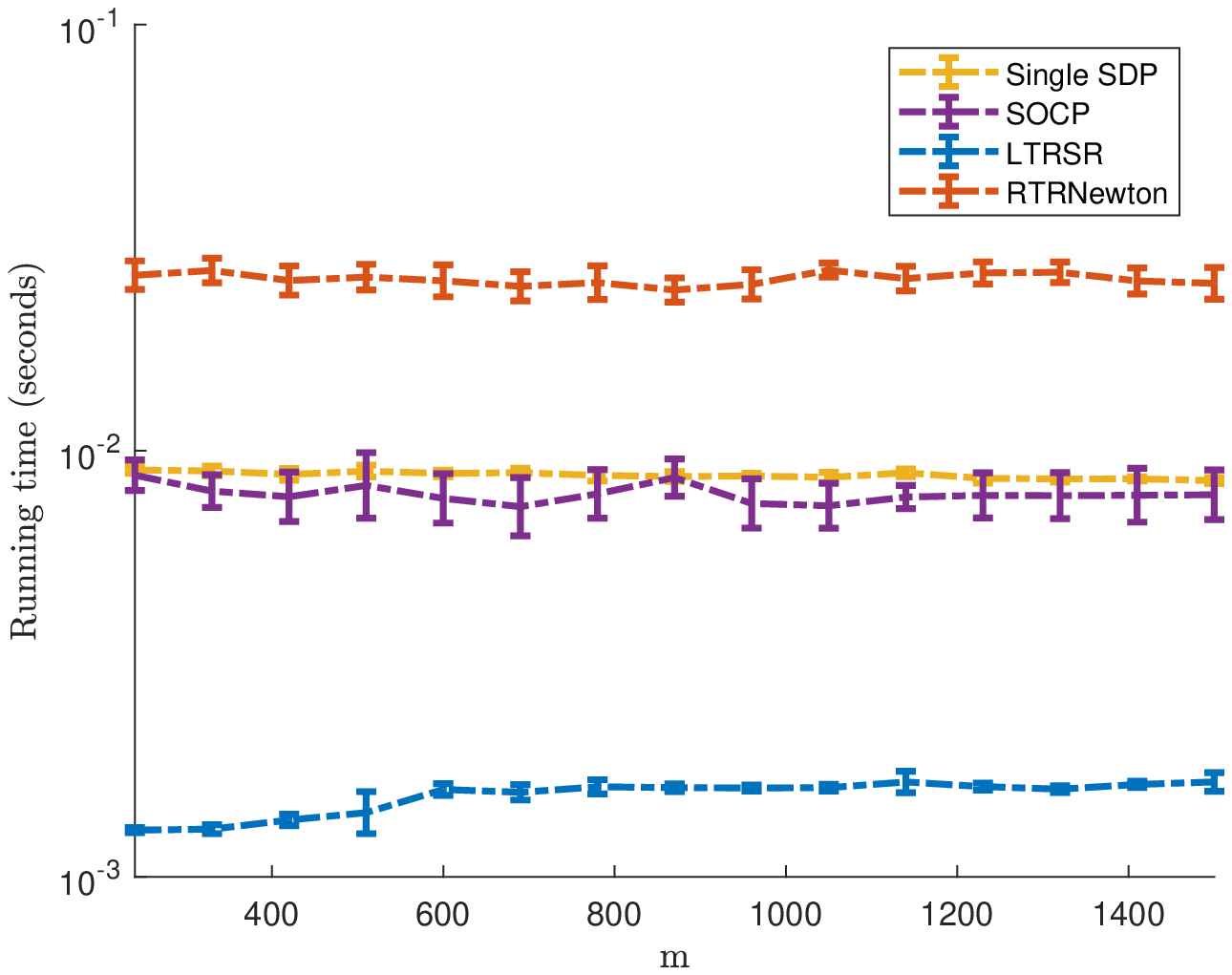} \label{fig:winemodestime}
}
\subfigure{
\includegraphics[scale=0.27]{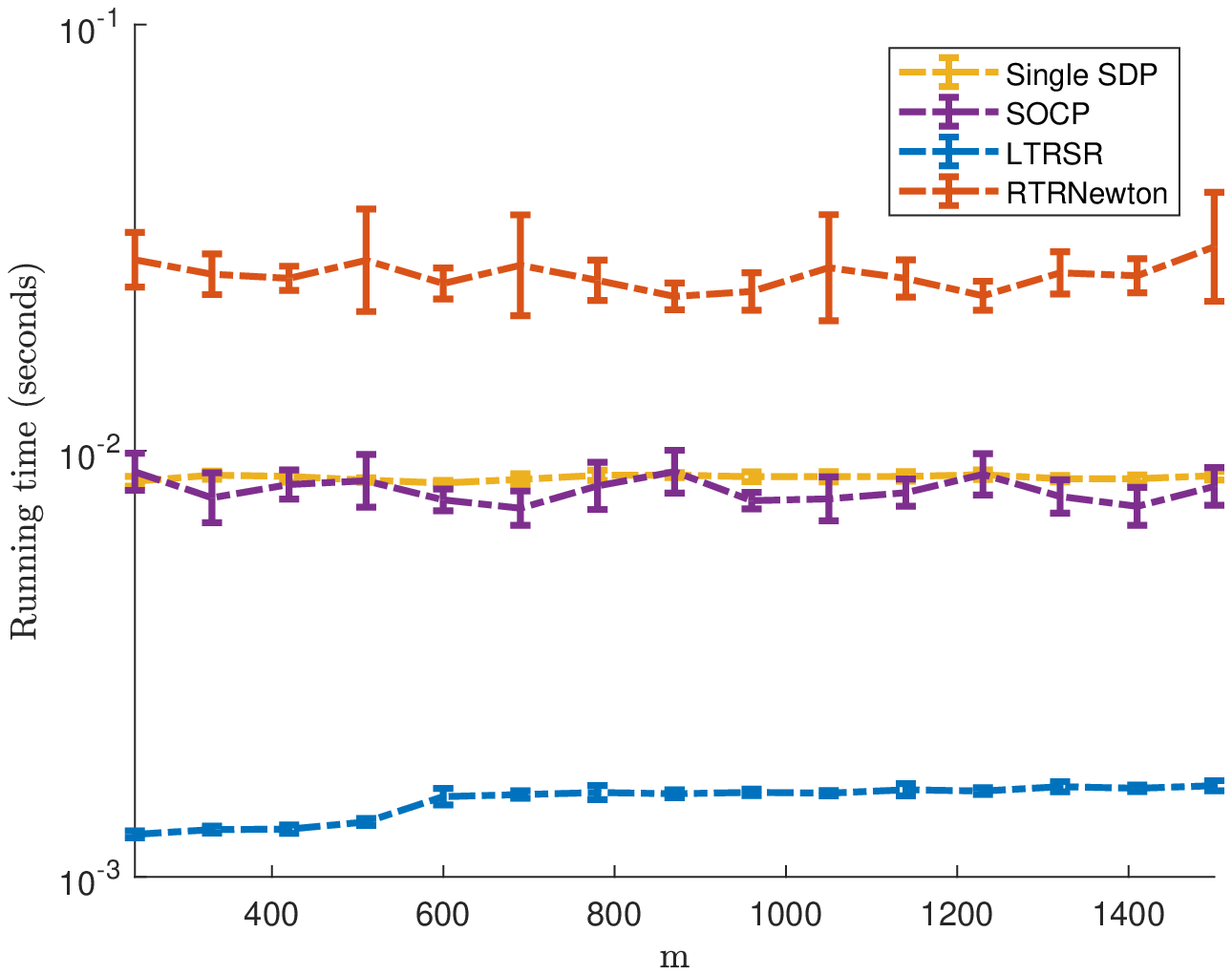} \label{fig:wineseveretime}}
\caption{Comparison of four different algorithms on the red wine dataset.  The left and right plots correspond to  $\mathcal A_{\text{modest}}$ and $\mathcal A_{\text{severe}}$, respectively.}
\label{fig:winetime}
\end{figure}
\begin{figure}[htbp]
\vspace{-0.1in}
\centering
\subfigure{
\includegraphics[scale=0.26]{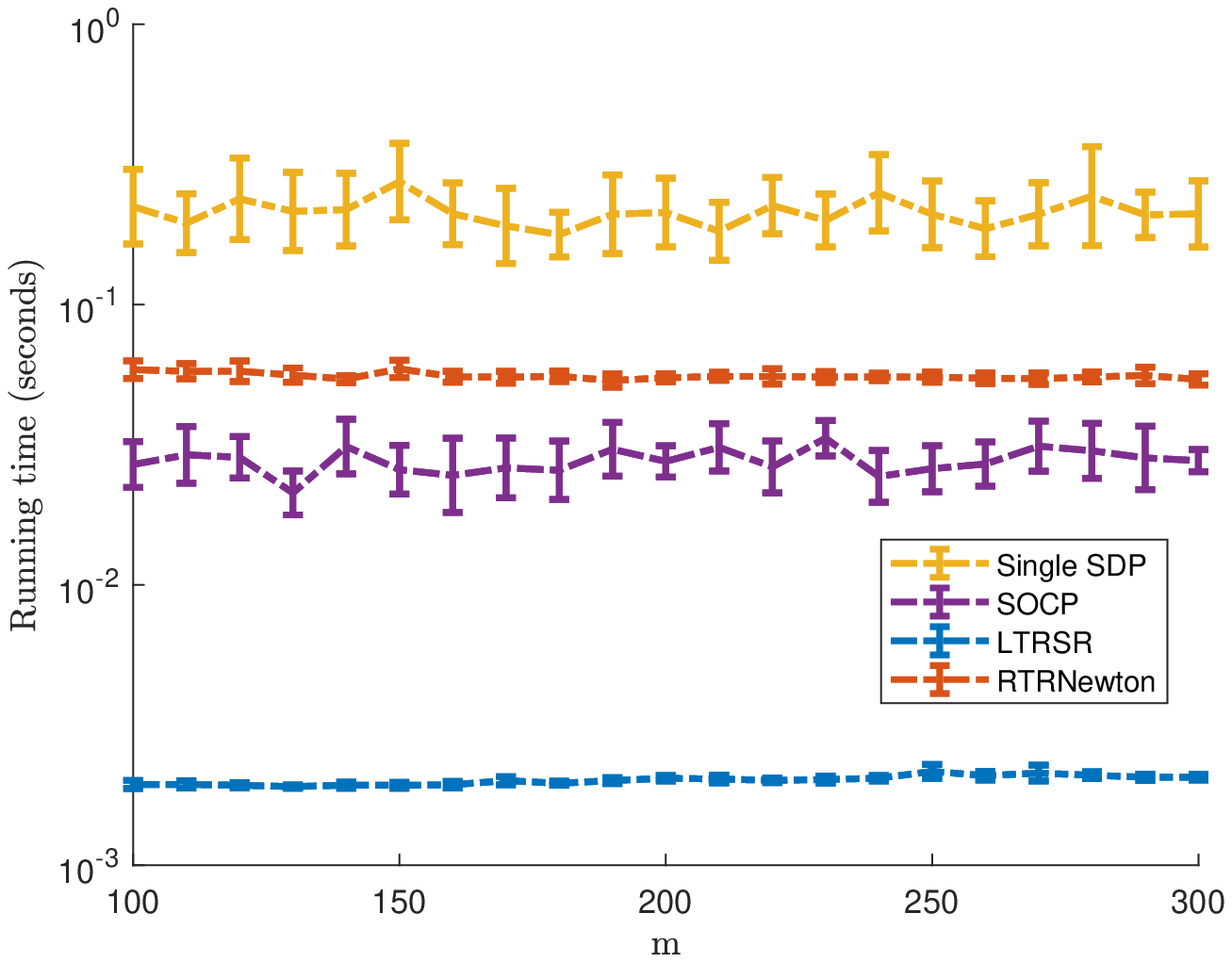} \label{fig:buildingmodestime}
}
\vspace{-0.1in}
\subfigure{
\includegraphics[scale=0.26]{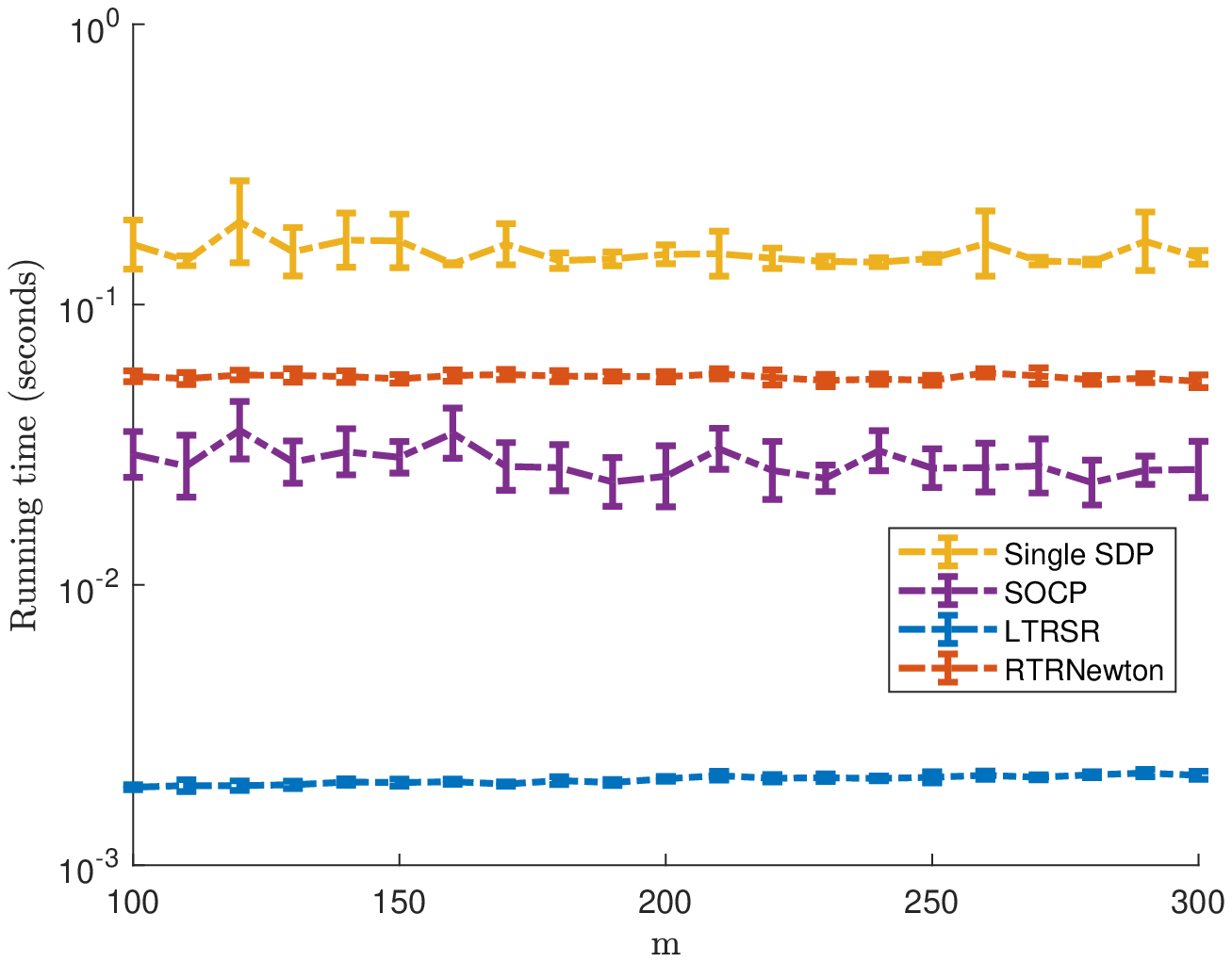} \label{fig:buildingseveretime}}
%\vspace{-0.1in}
\caption{Comparison of four different algorithms on the residential building dataset.  The left and right plots correspond to  $\mathcal A_{\text{modest}}$ and $\mathcal A_{\text{severe}}$, respectively.}
\label{fig:buildtime}
\end{figure}

We also compare four methods on the residental building dataset\footnote{\url{https://archive.ics.uci.edu/ml/datasets/Residential+Building+Data+Set}} from the UCI data repository \cite{Dua:2019} as well. The building dataset includes 372 samples with 107 features. Each feature reflects information of a certain session such as project date, physical and financial elements. The output label is the sale prices to real estate single-family residential apartments.
We consider a scenario where the seller wants to manipulate the price to higher level.
As a buyer, our task is to predict the fair price under fake data.
We still consider two types of sellers:  $\mathcal A_{\text{modest}}$ with  $ \delta = 20$  and  $\mathcal A_{\text{severe}}$ with $ \delta= 40 $.

The computational time for both datasets is reported in Figures \ref{fig:winetime} and  \ref{fig:buildtime}. It show that LTRSR method outperforms others and follows by the SOCP and RTRNewton. One can also observe that RTRNewton is even more expensive than the SDP approach in  Figure \ref{fig:winetime}.
The main reason is that in the red wine dataset, the number of features $n$ is quite small ($n=11$). Thus, the spectral decomposition step, as well as the iterations of the interior point method, in the SOCP approach is cheap.

We then report the relative errors of objective values (MSEs) of the SOCP method
and our methods in Table \ref{tab:SPGLS-real-reltol-part} for red wine and residental building datasets. Indeed, all the methods have a very high accuracy as the relative errors are only up to  \texttt{3.37e-5}. More MSE comparisons can be found in \cref{ap:reldifrealdata}.

\begin{table}[!htb]
	\setlength\tabcolsep{3pt}
	\centering
	\caption{Relative error of objective values }
\resizebox{\linewidth}{!}{ 	
\begin{tabular}{@{}lcccccc@{}}
		\toprule
		%		\multicolumn{13}{c}{$m=2n$}\\
		%		\midrule
			\multicolumn{1}{c}{\multirow{2}{*}{\bf{Dataset}}}
		&\multicolumn{3}{c}{ $\bf{(f_\textbf{SOCP}-f_\textbf{LTRSR})/|f_\textbf{SOCP}|}$}
		&\multicolumn{3}{c}{ $\bf{(f_\textbf{SOCP}-f_\textbf{RTRNew})/|f_\textbf{SOCP}|}$}
		\\
		\cmidrule(r){2-4} \cmidrule(r){5-7}
		%		&\textbf{AVG} & \textbf{MIN} & \textbf{MAX} & \textbf{AVG} & \textbf{MIN} & \textbf{MAX} \\
		& \multicolumn{1}{c}{\textbf{AVG}} & \multicolumn{1}{c}{\textbf{MIN}} & \multicolumn{1}{c}{\textbf{MAX}} & \multicolumn{1}{c}{\textbf{AVG}} & \multicolumn{1}{c}{\textbf{MIN}} & \multicolumn{1}{c}{\textbf{MAX}}  \\
		\midrule
		\textbf{Wine Modest} & 6.17E-10 & 3.94E-12 & 4.23E-09 & -8.26E-10
		&-4.66E-09 &3.62E-09                    \\
		\textbf{Wine Severe} &1.32E-10 & 3.39E-12 & 1.84E-09& -8.30E-11 &-4.07E-10 &1.80E-09 \\
		\textbf{Build Modest} & 1.49E-07 & 1.93E-09 & 5.91E-07 &-2.19E-05 &-3.37E-05 &-1.28E-05 \\
		\textbf{Build Severe} &3.02E-08 & 4.02E-10 & 1.25E-07 &-1.96E-06 &-3.06E-06 &-1.14E-06   \\		
		\bottomrule
			\label{tab:SPGLS-real-reltol-part}
	\end{tabular}
}
\end{table}%

\subsection{Synthetic Dataset}
%As evidenced from \cite{pmlr-v139-wang21d}, the SOCP approach is much faster than the SDP approach.
From the previous subsection, we also see that SOCP, LTRSR and RTRNewton are much faster than the SDP approach. To have a comprehensive comparison on wall-clock time among SOCP, LTRSR and RTRNewton, we test these methods on synthetic experiments.

\subsubsection{Dense Data}\label{sec:exp_syn_dense}

\begin{table}[ht]
%\vspace{-0.2in}
        %\small
        \setlength\tabcolsep{3pt}
        \centering
        \caption{Time (seconds) on synthetic data without sparsity}
        \resizebox{\linewidth}{!}{ 	
        \begin{tabular}{ccrccc}
                \toprule
                \multicolumn{6}{c}{$m=2n$}\vspace{-0.03in}\\
                \midrule
                $m$   & $n$ & SOCP (eig time)  & RTRNew & LTRSR  & Ratio \\%& R2 \\
                \midrule
            %    200 & 100 & 0.024 (0.002) & 0.027 & $\bm{0.003}$ & 8 & 9 \\
            %    1000 & 500 & 0.144 (0.016) & 0.277 & $\bm{0.008}$ & 18 & 35 \\
                2000 & 1000 & 0.585 (0.064) & 0.743 & $\bm{0.034}$ & 17 \\%& 22 \\
                4000 & 2000 & 1.957 (0.317) & 2.459 & $\bm{0.177}$ & 11 \\%& 14 \\
                8000 & 4000 & 10.693 (2.758) & 9.269 & $\bm{0.931}$ & 11 \\%& 10 \\
                12000 & 6000 & 29.304 (9.444) & 18.824 & $\bm{2.120}$ & 14 \\%& 9 \\
                16000 & 8000 & 58.561 (21.634) & 40.711 & $\bm{3.982}$ & 15 \\%& 10 \\
                20000 & 10000 & 114.376 (49.754) & 59.768 & $\bm{6.099}$ & 19 \\%& 10 \\
                \midrule

                \multicolumn{6}{c}{$m=n$} \vspace{-0.03in}\\
                \midrule
                $m$   & $n$ & SOCP (eig time)  & RTRNew & LTRSR  & Ratio \\%& R2 \\
                \midrule
                  %              100 & 100 & 0.072 (0.001) & 0.036 & $\bm{0.002}$ & 36 \\%& 18 \\
  %              500 & 500 & 0.137 (0.015) & 0.261 & $\bm{0.006}$ & 23 \\%& 44 \\
                1000 & 1000 & 0.454 (0.065) & 0.594 & $\bm{0.017}$ & 27 \\%& 35 \\
                2000 & 2000 & 2.104 (0.325) & 2.600 & $\bm{0.097}$ & 22 \\%& 27 \\
                4000 & 4000 & 10.795 (2.698) & 6.958 & $\bm{0.478}$ & 23 \\%& 15 \\
                6000 & 6000 & 28.391 (9.481) & 17.835 & $\bm{1.083}$ & 26\\% & 16 \\
                8000 & 8000 & 55.263 (21.555) & 35.510 & $\bm{2.011}$ & 27\\% & 18 \\
                10000 & 10000 & 97.383 (40.091) & 58.009&$\bm{3.065}$ & 32 \\%& 19 \\
                \midrule
                \multicolumn{6}{c}{$m=0.5n$}\vspace{-0.03in}\\
                \midrule
                $m$   & $n$ & SOCP (eig time)  & RTRNew & LTRSR  & Ratio \\%& R2 \\
                \midrule
                500 & 1000 & 0.536 (0.059) & 0.523 & $\bm{0.022}$ & 24 \\%& 24 \\
                1000 & 2000 & 1.748 (0.309) & 1.432 & $\bm{0.057}$ & 31 \\%& 25 \\
                2000 & 4000 & 9.928 (2.526) & 6.932 & $\bm{0.251}$ & 40 \\%& 28 \\
                3000 & 6000 & 27.230 (8.848) & 19.179 & $\bm{0.532}$ & 51 \\%& 36 \\
                4000 & 8000 & 54.174 (20.258) & 28.953 & $\bm{0.928}$ & 58 \\%& 31 \\
                5000 & 10000 & 94.548 (37.771) & 52.756 & $\bm{1.558}$ & 61 \\%& 34 \\
            \bottomrule
        \end{tabular}
        }
\label{tab:syn-dense}
\end{table}%
We first conduct experiments on dense synthetic dataset. To have better validation of the effectiveness of our proposed reformulation, we use the same artificial dataset in \citet{pmlr-v139-wang21d}, which employs  \texttt{make\_regression} function in scikit-learn~\cite{pedregosa2011scikit} with setting the noise as 0.1 and other parameters as default.

Table  \ref{tab:syn-dense} summarises the comparison of time on different scales with $m = ln,\ l \in \{0.5,1,2\}$. Here, ``SOCP" represents  total time needed for the SOCP approach (including ``eig time"), ``eig time”  represents the spectral decomposition time in the SOCP approach, ``RTRNew" represents the RTRNewton method, ``LTRSR" represents the LTRSR method, ``Ratio" represents the ratio of times of  SOCP method and LTRSR.

From Table \ref{tab:syn-dense}, we find that in large scale setting, the two methods RTRNewton and LTRSR are more efficient. LTRSR is of orders faster than the other two methods.
%For example, in the case $(m,n)= (10000,10000) $, LTRSR takes about $3.065$ seconds, while the RTRNewton and SOCP take $58.009$ seconds and $97.383$ seconds, respectively.
From the ``Ratio'' values we also see that the time cost of SOCP approach is several tens times of that of LTRSR. %Table \ref{tab:syn-dense} demonstrates the efficiency of our SCLS reformulation and the LTRSR methods for solving the SCLS.
We also observe that the spectral decomposition time in formulating SOCP is expensive and takes about 40\% of total time.
Indeed, the spectral decomposition time  becomes much larger as $n$ increases, which is also evidenced from our experiments for the sparse data setting in  Table \ref{tab:syn-sparse} below.

 \subsubsection{Sparse Data}\label{sec:exp_syn_sparse}
To further show the efficacy of our proposed reformulation, we conduct experiments on synthetic data with high feature dimension and various sparsity.  We apply the \texttt{sprandn} function in MATLAB to obtain the data matrix $X \in \mathbb{R}^{m\times n}$, whose $i$-th row is input vector $\{\x_i\}_{i=1}^m$.
%Specifically, we generate the parameter vector $\mathbf{\beta}$ i.i.d from the uniform distribution $[0,1]$.
The noise measurements $\{\xi_i\}_{i=1}^m$ i.i.d from the uniform distribution $[0, 0.5] $. Then the output label $\{y_i\}_{i=1}^m$ via
$
y_i = \x_i^T{\mathbb{\beta}} + \xi_i.
$ Following \citet{pmlr-v139-wang21d}, we set the fake output label  as $z_i = \max\{y_i,y_{0.25}\}$.

\begin{table}[h]
%\vspace{-0.1in}
       % \small
        \setlength\tabcolsep{3pt}
        \centering
        \caption{Time (seconds) on synthetic data with sparsity}
        \resizebox{\linewidth}{!}{ 	
        \begin{tabular}{ccrccc}
            \toprule
            \multicolumn{6}{c}{sparsity = 0.01}\vspace{-0.01in}\\
            \midrule
            $m$   & $n$ & SOCP (eig time)  & RTRNew & LTRSR  & Ratio \\%& R2 \\
            \midrule
          % 500 & 1000 & 0.289 (0.068) & 0.100 & $\bm{0.023}$ & 13 & 4 \\
          % 1000 & 2000 & 1.065 (0.404) & 0.311 & $\bm{0.013}$ & 82 & 24 \\
          % 2000 & 4000 & 6.213 (2.851) & 1.270 & $\bm{0.032}$ & 194 & 40 \\
          % 3000 & 6000 & 18.745 (9.484) & 3.932 & $\bm{0.081}$ & 231 & 49 \\
          % 4000 & 8000 & 38.894 (21.259) & 6.249 & $\bm{0.145}$ & 268 & 43 \\
           5000 & 10000 & 71.601 (39.432) & 13.124 & $\bm{0.225}$ & 318 \\%& 58 \\
           7500 & 15000 & 217.529 (120.456) & 26.551 & $\bm{0.534}$ & 407 \\%& 50 \\
           10000 & 20000 & 513.751 (288.490) & 47.411 & $\bm{1.049}$ & 490 \\%& 45 \\
           12500 & 25000 & 941.394 (539.619) & 69.421 & $\bm{1.606}$ & 586 \\%& 43 \\
           15000 & 30000 & 1539.443 (865.813) & 113.223 & $\bm{2.416}$ & 637 \\%& 47 \\
            \midrule

            \multicolumn{6}{c}{sparsity = 0.001} \vspace{0.01in}\\
            \midrule
            $m$   & $n$ & SOCP (eig time)  & RTRNew & LTRSR  & Ratio \\%& R2 \\
            \midrule
   %        500 & 1000 & 0.557 (0.087) & 0.086 & $\bm{0.006}$ & 93 & 14 \\
   %        1000 & 2000 & 1.169 (0.418) & 0.175 & $\bm{0.005}$ & 234 & 35 \\
   %        2000 & 4000 & 4.340 (3.182) & 0.390 & $\bm{0.007}$ & 620 & 56 \\
   %        3000 & 6000 & 14.943 (10.974) & 0.747 & $\bm{0.013}$ & 1150 & 58 \\
   %        4000 & 8000 & 33.700 (25.250) & 1.042 & $\bm{0.020}$ & 1685 & 52 \\
           5000 & 10000 & 61.587 (45.253) & 1.416 & $\bm{0.028}$ & 2200 \\%& 51 \\
           7500 & 15000 & 153.075 (117.389) & 2.379 & $\bm{0.053}$ & 2888 \\%& 45 \\
           10000 & 20000 & 335.956 (259.671) & 5.453 & $\bm{0.113}$ & 2973 \\%& 48 \\
           12500 & 25000 & 638.175 (491.391) & 7.715 & $\bm{0.168}$ & 3799 \\%& 46 \\
           15000 & 30000 & 1082.261 (832.413) & 12.090 & $\bm{0.235}$ & 4605 \\%& 51 \\
            \midrule
            \multicolumn{6}{c}{sparsity = 0.0001}\vspace{-0.01in}\\
            \midrule
            $m$   & $n$ & SOCP (eig time)  & RTRNew & LTRSR  & Ratio \\%& R2 \\
            \midrule
   %        500 & 1000 & 0.190 (0.073) & 0.053 & $\bm{0.003}$ & 63 & 18 \\
   %        1000 & 2000 & 0.579 (0.392) & 0.117 & $\bm{0.004}$ & 145 & 29 \\
   %        2000 & 4000 & 4.284 (3.160) & 0.260 & $\bm{0.006}$ & 714 & 43 \\
   %        3000 & 6000 & 13.059 (10.905) & 0.242 & $\bm{0.008}$ & 1632 & 30 \\
   %        4000 & 8000 & 26.882 (25.215) & 0.439 & $\bm{0.010}$ & 2688 & 44 \\
           5000 & 10000 & 49.869 (45.462) & 0.391 & $\bm{0.009}$ & 5541 \\%& 43 \\
           7500 & 15000 & 141.507 (134.525) & 0.716 & $\bm{0.014}$ & 10108 \\%& 51 \\
           10000 & 20000 & 310.991 (289.447) & 0.979 & $\bm{0.020}$ & 15550 \\%& 49 \\
           12500 & 25000 & 587.124 (540.314) & 1.301 & $\bm{0.030}$ & 19571 \\%& 43 \\
           15000 & 30000 & 991.070 (912.015) & 2.171 & $\bm{0.037}$ & 26786 \\%& 59 \\
        \bottomrule
    \end{tabular}
    }
\label{tab:syn-sparse}
\end{table}%

Table \ref{tab:syn-sparse} summarises  time comparisons on  synthetic datasets with different sparsity and different dimension for $m = 0.5n$.
From these tables, we find that  LTRSR and RTRNewton perform much better than the dense case, and their superiority over the SOCP approach becomes larger. This is mainly because both methods are the matrix free methods that require only matrix vector products in each iteration.
However, the SOCP do not benefit from sparsity as well as the other two methods. We find that,  by comparing the "eig time" for different instances with the same dimension but different sparsity, the "eig time" dominates the time of SOCP approach as the spectral decomposition cannot utilize the sparsity of the data.
From the ``Ratio" values, we find that the outperformance of LTRSR grows considerably when the sparsity and problem size increase. In the case of $(m,n)= (15000,30000)$ and sparsity = 0.0001, LTRSR takes up to 26,000+ times faster than the SOCP approach. Moreover, our LTRSR takes less than 0.05 second for all the instances with  sparsity = 0.0001.
More reports on relative errors of all methods are reported in \cref{ap:reldiffsyn}.
\section{Conclusion}
We propose an SCLS reformulation for the SPG-LS and show its optimal solution can be used to recover an optimal solution of the SPG-LS.
We further show that an $\epsilon$ optimal solution of the SPG-LS can be also recovered from an $\epsilon$ optimal solution of the SCLS.
Moreover, such an $\epsilon$ optimal solution  obtained in runtime $\tO(N/\sqrt\epsilon)$.
We also introduce two practical efficient methods, LTRSR and RTRNewton, for solving the SCLS.
Experiments show that the SCLS approach is much faster than the existing best approach.
In particular, the performance of the LTRSR dominates both RTRNewton and SOCP methods.

\section*{Acknowledgements}
Wen Huang is partly supported by the Fundamental Research Funds for the Central Universities 20720190060 and NSFC 12001455.
Rujun Jiang is partly supported by NSFC 12171100 and 72161160340, and Natural Science Foundation of Shanghai 22ZR1405100. 
Xudong Li is partly supported by the ``Chenguang Program" by Shanghai Education Development Foundation and Shanghai Municipal Education Commission 19CG02, and the Shanghai Science and Technology Program 21JC1400600.

%%%%%%%%%%%%%%%%%%%%%%%%%%%%%%%%%%%%%%%%%%%%%%%%%%%%%%%%%%%%
% \bibliography{ref}
% \bibliographystyle{icml2022}

%\appendix
%\newpage
%
%
%%\onecolumn
%\begin{center}
%{\Large \bf Supplementary Material}
%\end{center}
%\par\noindent\rule{\textwidth}{1pt}
%\setcounter{section}{0}
\appendix
\newpage
\onecolumn
\begin{center}
{\Large \bf Appendix}
\end{center}
\vspace{-0.05in}
\par\noindent\rule{\textwidth}{1pt}
\setcounter{section}{0}

\renewcommand\thesection{\Alph{section}}
\renewcommand\thesubsection{\arabic{subsection}}

\section{Deferred proof for \cref{thm:RTRN}}\label{ap:RTRN}
\begin{proof}
We only need to verify the assumptions of Theorem~7.4.4 and Theorem~7.4.11 of~\citet{AMS2008}.

By definition, the function $q$ is bounded below. Since the unit sphere is a compact manifold,
we have that the Riemannian Hessian $\Hess\; q(\br)$ is bounded above for any $\br \in \mathcal{S}^{n}$,
that $q\circ R$ is radially Lipschitz continuously differentiable, and that $q$ is Lipschitz continuously
differentiable by Proposition~7.4.5 and Corollary~7.4.6  of~\citet{AMS2008}.
The implementation in ROPTLIB for the subproblem (17) is Algorithm~11 of~\citet{AMS2008} with $\theta = 1$.
Therefore, the Cauchy decrease inequality is satisfied. It follows that all the assumptions of Theorem~7.4.4
are satisfied.

Since Retraction~\eqref{Retr} is second order, the assumption of~(7.36) in Theorem~7.4.11 of~\citet{AMS2008}
holds with $\beta_{\mathcal{H}} = 0$.
Since the function~$q$ is $C^{\infty}$ and the manifold is compact, the assumption of
(7.37) in Theorem~7.4.11 of~\citet{AMS2008} holds. The local quadratic convergence rate in~\eqref{rtrquadratic} thus follows.
\end{proof}
\section{Additional Experiments}

In this section, we provide  additional numerical results to further show the efficiency of our proposed reformulation.
\subsection{Real-world Dataset}
We demonstrate the speed of our methods on two other real-world datasets, the insurance dataset\footnote{\url{https://www.kaggle.com/mirichoi0218/insurance}} and the blogfeedback dataset\footnote{\url{https://archive.ics.uci.edu/ml/datasets/BlogFeedback}}. Similar as the setting in previous section, we compare the wall-clock time of RTRNewton \cite{absil2007trust} and LTRSR \cite{zhang2018nested} approaches with the SDP and SOCP methods proposed in \citet{pmlr-v139-wang21d}.

We still apply four methods on the insurance dataset that has 1,338 instances with 7 features. Each feature shows information on certain aspects such as age, sex, bmi and region. The output labels are individual medical costs by buying health insurance. For model accuracy, we transform  categorical features such as sex into a one-hot vector. We assume that the individuals incline to modify self-related data to reduce their insurance costs. Formally, the individual's desired outcome can be defined as
$
z_i = \max \{y_i + \delta,\ 0\}.
$
We have two types of individual: $\mathcal A_{\text{modest}}$ with  $ \delta = -100$  and  $\mathcal A_{\text{severe}}$ with $ \delta= -300 $. All the hyperparameters are the same as those in \citet{pmlr-v139-wang21d}.
As an insurer, our goal is to select a good price model to predict the insurance costs as true as possible.

To further illustrate the effectiveness of our new reformulation, we compare four methods on the blogfeedback dataset. The blog dataset contains 52,397 samples each with 281 features processed from raw feedback materials on the Internet. Each feature represents information of a certain session. The output label is the number of comments. As a learner, our task is to predict the future comments of blog.
Similarly, we assume that the true output label $\y$ would be modified by data providers in order to achieve the goal of increasing blog comments. For example, public media intend to manipulate data to add the blog comments and  enhance its news popularity.
 Formally, we define the altered label as $ z_i = y_i + \delta$. We still have two types of data providers:  $\mathcal A_{\text{modest}}$ with  $ \delta = 5$  and  $\mathcal A_{\text{severe}}$ with $ \delta= 10 $.

The wall-clock time comparison can be found in Figure \ref{fig:bloginsurtime}. Similar to the previous cases, LTRSR outperforms other approaches. However, as the dimension in this problem is too small, the comparison of the other three methods does not match their performance for large scale problems.

\begin{figure}[htbp]
\centering

\subfigure{
\includegraphics[scale=0.28]{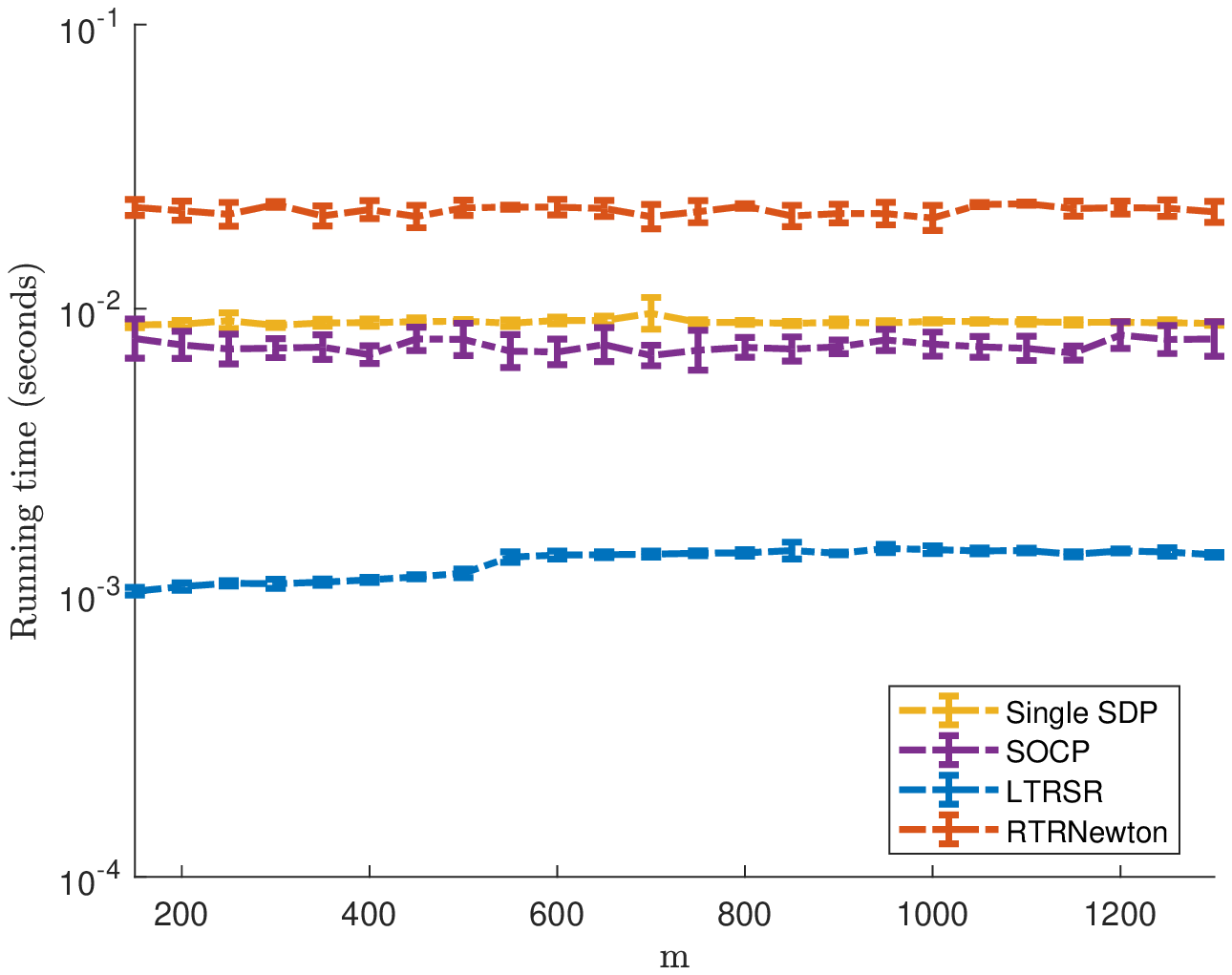} \label{fig:insurmodestime}
}
\subfigure{
\includegraphics[scale=0.28]{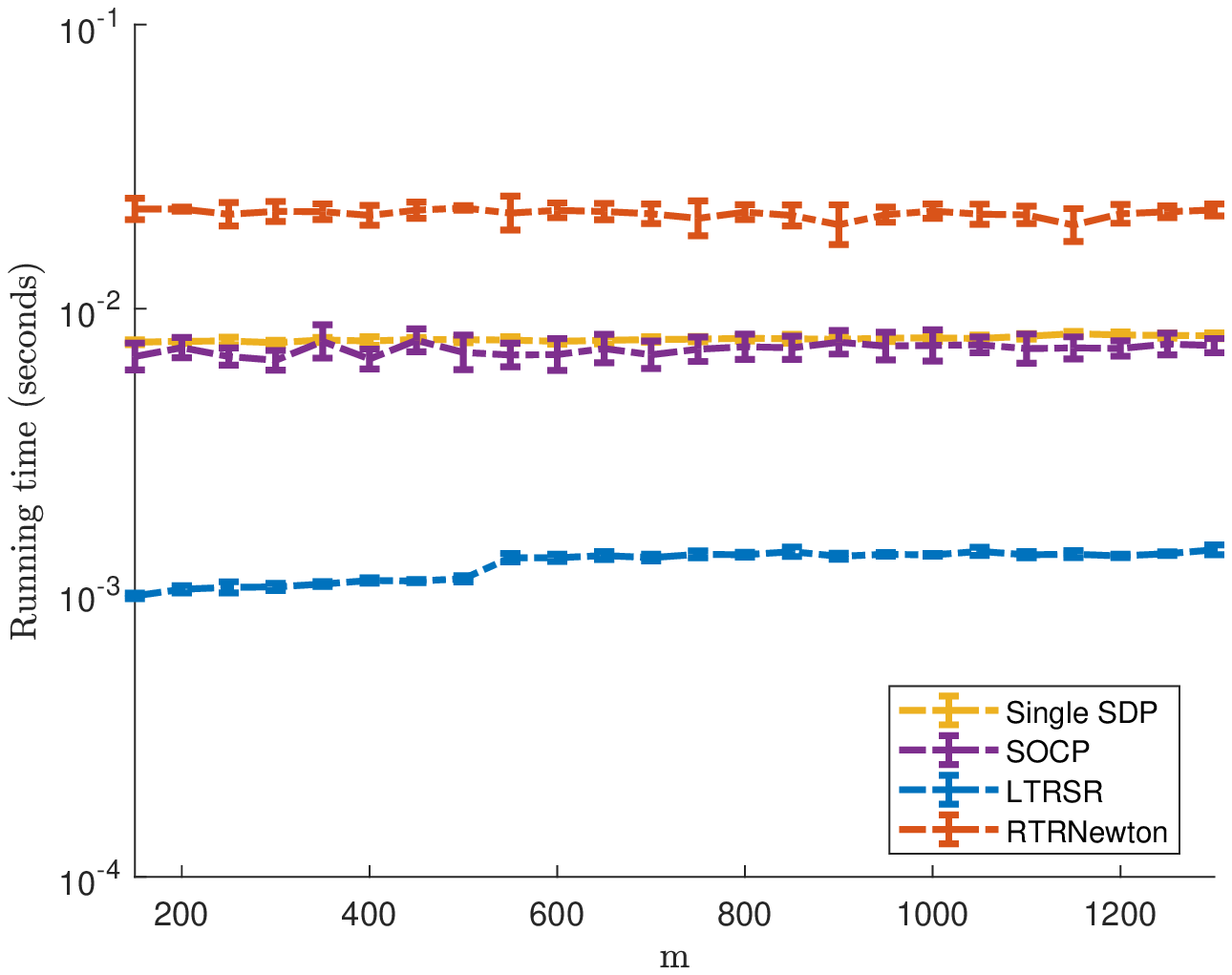} \label{fig:insurseveretime}}
\subfigure{
\includegraphics[scale=0.28]{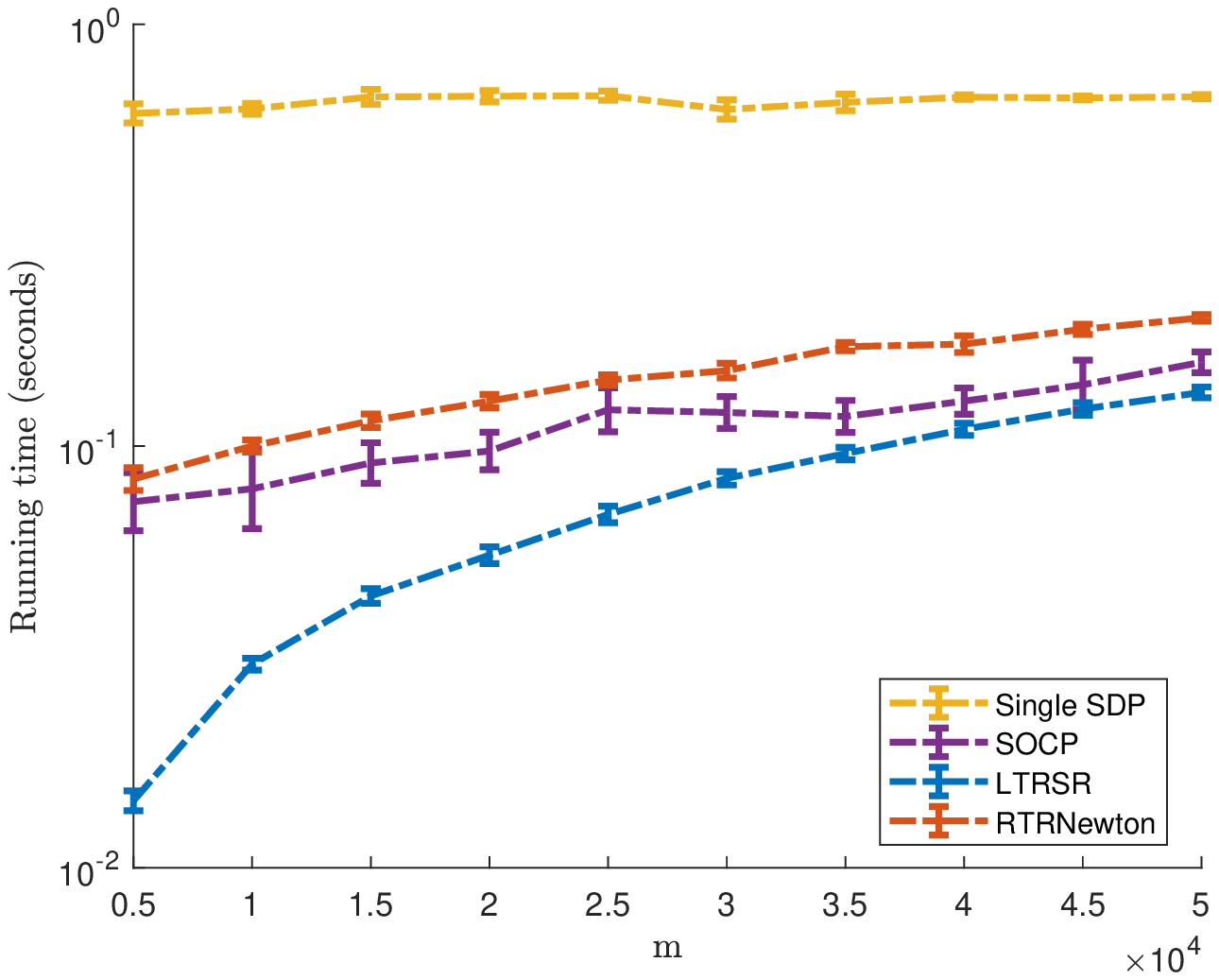} \label{fig:blogmodestime}
}
\subfigure{
\includegraphics[scale=0.28]{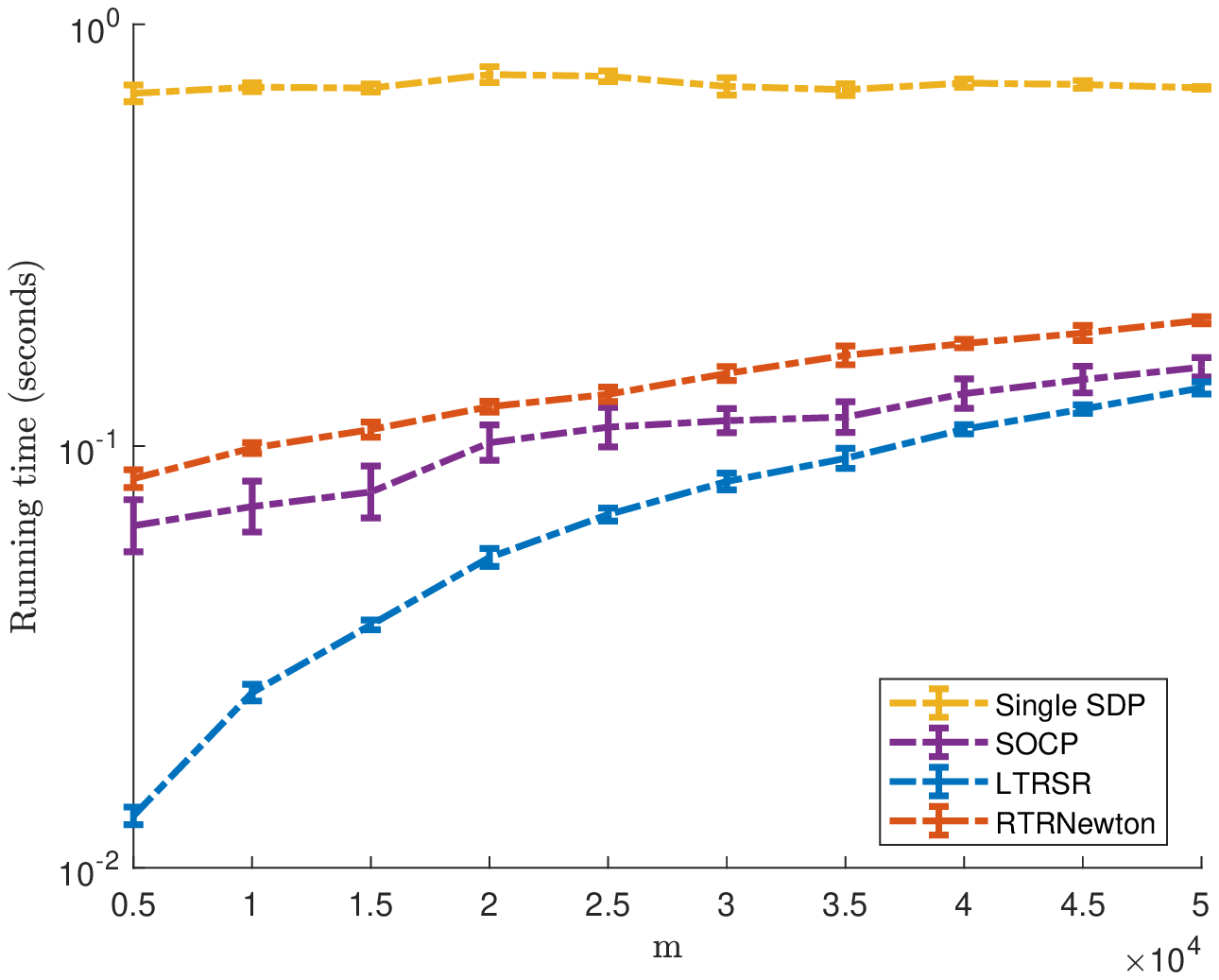} \label{fig:blogseveretime}}
\caption{Performance comparison between different algorithms on the insurance and blog dataset.  The left two plots correspond to the wall-clock time comparison of insurance dataset generated  by $\mathcal A_{\text{modest}}$ and $\mathcal A_{\text{severe}}$, whilst the right two plots correspond to the wall-clock time comparison of blog dataset generated by $\mathcal A_{\text{modest}}$ and $\mathcal A_{\text{severe}}$.}
\label{fig:bloginsurtime}
\end{figure}
\subsection{Synthetic Dataset}
To further demonstrate the efficiency and adaptiveness to high dimension problems of our proposed reformulation, we conduct experiments on more synthetic datasets with different hyperparameters and  various dimensions and sparsity.
\subsubsection{Dense Data}

We first perform experiments with $\gamma=0.01$ on synthetic datasets without sparsity to show the superiority of our reformulation under different hyperparameter.
 All the other settings are the same as in Section \ref{sec:exp_syn_dense}.

Table \ref{tab:app-syn-dense} shows the comparison of wall-clock time on different scales with $m = ln,\ l \in \{0.5,1,2\}$, which is similar to the case of $\gamma=0.1$. We observed that both  SCLS approaches are faster than the SOCP approach. Moreover, all the LTRSR is less than 6 seconds while the SOCP can take up to about 110 seconds.

\begin{table}[!h]
\setlength\tabcolsep{3pt}
\centering
\caption{Time (seconds) on synthetic data without sparsity, $\gamma=0.01$}
\resizebox{\linewidth}{!}{
\begin{tabular}{ccccc|cccc|cccc}
% \begin{tabular}{@{}c|cccc|cccc|cccc@{}}
	\toprule
    \multirow{2}{*}{ $n$}
    &\multicolumn{4}{c|}{$m =2n$}
    &\multicolumn{4}{c|}{$m =n$}
    &\multicolumn{4}{c}{$m =0.5n$} \\
    \cmidrule(r){2-5} \cmidrule(r){6-9} \cmidrule(r){10-13}
	
      &  SOCP (eig time)  & RTRNew   & LTRSR  & Ratio
      &  SOCP (eig time)  & RTRNew   & LTRSR  & Ratio
      &  SOCP (eig time)  & RTRNew   & LTRSR  & Ratio  \\
    \midrule
     1000 & 0.619 (0.087) & 0.712 & $\bm{0.031}$ & 20  & 0.564 (0.086) & 0.704 & $\bm{0.020}$ & 28   & 0.466 (0.078) & 0.649 & $\bm{0.012}$ & 39\\%& 23 \\
     2000 & 2.244 (0.419) & 2.519 & $\bm{0.138}$ & 16  & 1.900 (0.471) & 1.828 & $\bm{0.098}$ & 19  & 2.438 (0.401) & 2.092 & $\bm{0.060}$ & 41\\%& 18 \\
     4000 & 12.123 (3.448) & 5.179 & $\bm{0.956}$ & 13  & 11.597 (3.539) & 6.789 & $\bm{0.499}$ & 23  & 11.645 (3.212) & 6.778 & $\bm{0.249}$ & 47\\%& 5 \\
     6000 & 33.093 (11.903) & 15.857 & $\bm{2.135}$ & 16  & 31.262 (11.691) & 18.617 & $\bm{1.053}$ & 30  & 32.812 (11.008) & 11.070 & $\bm{0.561}$ & 58\\%& 7 \\
     8000 & 66.816 (27.466) & 38.501 & $\bm{3.768}$ & 18  & 63.983 (27.512) & 34.655 & $\bm{1.984}$ & 32  & 59.725 (25.061) & 27.201 & $\bm{0.961}$ & 62\\%& 10 \\
     10000 & 118.044 (49.477) & 59.551 & $\bm{5.916}$ & 20  & 109.516 (50.018) & 54.060 & $\bm{3.048}$ & 36  & 104.251 (47.261) & 39.611 & $\bm{1.529}$ & 68\\
	\bottomrule
		\label{tab:app-syn-dense}
\end{tabular}
}
\end{table}%

\subsubsection{Sparse Data}

\begin{table}[!thb]

%\vspace{-0.2in}

\setlength\tabcolsep{2pt}
\centering
\caption{Time (seconds) on synthetic data with different sparsity}
\resizebox{\linewidth}{!}{
\begin{tabular}{@{}ccccc|cccc|cccc@{}}
	\toprule
	\multicolumn{13}{c}{$m=n$}\\
	\midrule
	\multirow{2}{*}{$n$}
	&\multicolumn{4}{c|}{sparsity = 0.01}
	&\multicolumn{4}{c|}{sparsity = 0.001}
	&\multicolumn{4}{c}{sparsity = 0.0001} \\
	\cmidrule(r){2-5} \cmidrule(r){6-9} \cmidrule(r){10-13}
	
	&  SOCP (eig time)  & RTRNew   & LTRSR  & Ratio
	&  SOCP (eig time)  & RTRNew   & LTRSR  & Ratio
	&  SOCP (eig time)  & RTRNew   & LTRSR  & Ratio  \\
	\midrule
	10000  & 84.547 (40.602) & 21.947 & $\bm{0.487}$ & 174   & 56.101 (40.322) & 1.892 & $\bm{0.061}$ & 920   & 41.854 (38.038) & 0.499 & $\bm{0.012}$ & 3488  \\
	15000  & 254.85 (126.365) & 50.797 & $\bm{1.142}$ & 223   & 160.486 (125.064) & 5.166 & $\bm{0.131}$ & 1225   & 130.071 (118.582) & 0.540 & $\bm{0.018}$ & 7226  \\
	20000  & 550.191 (281.018) & 107.358 & $\bm{2.088}$ & 264   & 355.746 (278.687) & 10.932 & $\bm{0.222}$ & 1603   & 299.496 (265.987) & 1.075 & $\bm{0.029}$ & 10327  \\
	25000  & 1090.231 (581.289) & 159.592 & $\bm{3.502}$ & 311   & 728.118 (560.429) & 15.543 & $\bm{0.357}$ & 2040   & 591.172 (509.912) & 1.285 & $\bm{0.038}$ & 15557  \\
	30000  & 1726.133 (963.081) & 248.325 & $\bm{5.393}$ & 320   & 1240.319 (979.66) & 19.947 & $\bm{0.528}$ & 2349   & 1040.441 (888.321) & 2.323 & $\bm{0.056}$ & 18579  \\
	\midrule

	\multicolumn{13}{c}{$m=2n$}\\
	\midrule
	\multirow{2}{*}{	$n$}
	&\multicolumn{4}{c|}{sparsity = 0.01}
	&\multicolumn{4}{c|}{sparsity = 0.001}
	&\multicolumn{4}{c}{sparsity = 0.0001} \\
	\cmidrule(r){2-5} \cmidrule(r){6-9} \cmidrule(r){10-13}

	&  SOCP (eig time)  & RTRNew   & LTRSR  & Ratio
	&  SOCP (eig time)  & RTRNew   & LTRSR  & Ratio
	&  SOCP (eig time)  & RTRNew   & LTRSR  & Ratio  \\
	\midrule

	10000 & 108.861 (48.901) & 50.508 & $\bm{1.096}$ & 99  & 64.521 (48.604) & 5.765 & $\bm{0.119}$ & 542 & 45.599 (39.977) & 0.326 & $\bm{0.016}$ & 2850\\
	15000 & 316.424 (151.406) & 117.217 & $\bm{2.609}$ & 121  & 173.231 (132.098) & 12.230 & $\bm{0.257}$ & 674 & 148.608 (127.759) & 0.947 & $\bm{0.028}$ & 5307\\
	20000 & 643.464 (324.073) & 219.622 & $\bm{5.056}$ & 127  & 379.145 (289.701) & 23.199 & $\bm{0.457}$ & 830 & 338.119 (283.882) & 1.241 & $\bm{0.048}$ & 7044\\
	25000 & 1149.663 (605.581) & 395.202 & $\bm{7.818}$ & 147  & 720.837 (561.652) & 43.718 & $\bm{0.725}$ & 994 & 684.654 (563.812) & 2.254 & $\bm{0.061}$ & 11224\\
	30000 & 2026.088 (1080.883) & 598.468 & $\bm{12.022}$ & 169  & 1217.779 (937.93) & 45.074 & $\bm{1.100}$ & 1107 & 1106.028 (899.099) & 3.869 & $\bm{0.083}$ & 13326\\
	 \midrule

	 \multicolumn{13}{c}{$m=3n$}\\
	\midrule
	\multirow{2}{*}{$n$}
	&\multicolumn{4}{c|}{sparsity = 0.01}
	&\multicolumn{4}{c|}{sparsity = 0.001}
	&\multicolumn{4}{c}{sparsity = 0.0001} \\
	\cmidrule(r){2-5} \cmidrule(r){6-9} \cmidrule(r){10-13}
	
	&  SOCP (eig time)  & RTRNew   & LTRSR  & Ratio
	&  SOCP (eig time)  & RTRNew   & LTRSR  & Ratio
	&  SOCP (eig time)  & RTRNew   & LTRSR  & Ratio  \\
	\midrule
	10000 & 113.978 (51.388) & 76.183 & $\bm{1.592}$ & 72  & 56.434 (41.871) & 8.557 & $\bm{0.171}$ & 330 & 46.739 (39.697) & 0.446 & $\bm{0.022}$ & 2125  \\
	15000 & 298.498 (143.174) & 200.29 & $\bm{3.932}$ & 76  & 171.902 (129.377) & 22.335 & $\bm{0.396}$ & 434 & 150.666 (123.946) & 0.774 & $\bm{0.037}$ & 4072  \\
	20000 & 596.182 (288.772) & 356.849 & $\bm{7.750}$ & 77  & 417.155 (312.681) & 42.438 & $\bm{0.718}$ & 581 & 339.114 (274.737) & 1.713 & $\bm{0.061}$ & 5559  \\
	25000 & 1152.428 (606.097) & 614.192 & $\bm{13.592}$ & 85  & 782.197 (587.333) & 43.646 & $\bm{1.151}$ & 680 & 641.319 (518.51) & 2.678 & $\bm{0.085}$ & 7545  \\
	30000 & 1949.035 (1039.845) & 948.449 & $\bm{21.749}$ & 90  & 1298.734 (971.861) & 93.219 & $\bm{1.698}$ & 765 & 1085.053 (876.409) & 5.526 & $\bm{0.119}$ & 9118  \\	
	\bottomrule
	\label{tab:app-syn-sparse}
\end{tabular}
}
\end{table}%

We proceed to perform experiments on large-scale sparse dataset of various scales  $m = ln,\ l \in \{1,2,3\}$ with different sparsity. % to valid the efficiency of our SCLS approach.
All the other settings are the same as in Section \ref{sec:exp_syn_sparse}.

From Table \ref{tab:app-syn-sparse}, we observed the great superiority of our SCLS reformulation since all the LTRSR and RTRNewton method faster than SOCP method.
% especially the LTRSR method.
LTRSR is of several orders faster than the SOCP approach, especially when the sparsity and dimension grow.
%For example, in the case $(m,n)= (30000,30000)$, sparsity$ = 0.0001$, the LTRSR takes only about 1/18000 times of SOCP. Besides,
The spectral time of decomposition in formulating SOCP is quite expensive as the problem size grows. In the case $(m,n)= (90000,30000)$, sparsity $ = 0.01$, the decomposition time is up about 1000 seconds, while the LTRSR method only takes about 22 seconds to solve the problem.

\section{Relative Error}
This section reports  relative errors of function values and MSEs between SOCP
and our methods for all our experiments in Tables \ref{tab:SPGLS-real-reltol}, \ref{tab:SPGLS-real-MSE},  \ref{tab:SPGLS-dense-reltol} and \ref{tab:SPGLS-sparse-reltol}.
\subsection{Real-world Dataset}
In Tables \ref{tab:SPGLS-real-reltol} and \ref{tab:SPGLS-real-MSE}, we show the relative errors of MSEs on training sets and  test sets, respectively. Here, abbreviations ``Insur" represents insurance dataset, ``Build" represents building dataset, ``f" represents the function value of related methods and ``M" represents its MSE. In Table \ref{tab:SPGLS-real-reltol}, We observe that LTRSR is more accurate than SOCP as its relative error preserves positive. Indeed, all the methods
have a very high accuracy as the relative errors are only up to to \texttt{3.37e-5}.  Table \ref{tab:SPGLS-real-MSE} shows that all methods have similar test accuracy as the relative errors of the test MSEs of all the methods are up to \texttt{5.27e-4}.

\begin{table}[!h]
	\setlength\tabcolsep{5pt}
	\centering
	\caption{Relative error of objective values on training sets }
\resizebox{\linewidth}{!}{
	\begin{tabular}{@{}lcccccclcccccc@{}}
		\toprule
		%		\multicolumn{13}{c}{$m=2n$}\\
		%		\midrule
		\multicolumn{1}{c}{\multirow{2}{*}{\bf{{Dataset}}}}
		&\multicolumn{3}{c}{ $\bf{(f_\textbf{SOCP}-f_\textbf{LTRSR})/|f_\textbf{SOCP}|}$}
		&\multicolumn{3}{c}{ $\bf{(f_\textbf{SOCP}-f_\textbf{RTRNew})/|f_\textbf{SOCP}|}$}
		&\multicolumn{1}{c}{\multirow{2}{*}{\bf{Dataset}}}	
		&\multicolumn{3}{c}{ $\bf{(f_\textbf{SOCP}-f_\textbf{LTRSR})/|f_\textbf{SOCP}|}$}
		&\multicolumn{3}{c}{ $\bf{(f_\textbf{SOCP}-f_\textbf{RTRNew})/|f_\textbf{SOCP}|}$}
		\\
		\cmidrule(r){2-4} \cmidrule(r){5-7} \cmidrule(r){9-11} \cmidrule(r){12-14}
		%		&\textbf{AVG} & \textbf{MIN} & \textbf{MAX} & \textbf{AVG} & \textbf{MIN} & \textbf{MAX} \\
		& \multicolumn{1}{c}{\textbf{AVG}} & \multicolumn{1}{c}{\textbf{MIN}} & \multicolumn{1}{c}{\textbf{MAX}} & \multicolumn{1}{c}{\textbf{AVG}} & \multicolumn{1}{c}{\textbf{MIN}} & \multicolumn{1}{c}{\textbf{MAX}} & & \multicolumn{1}{c}{\textbf{AVG}} & \multicolumn{1}{c}{\textbf{MIN}} & \multicolumn{1}{c}{\textbf{MAX}} & \multicolumn{1}{c}{\textbf{AVG}} & \multicolumn{1}{c}{\textbf{MIN}} & \multicolumn{1}{c}{\textbf{MAX}} \\ \midrule

		\textbf{Wine Modest} & 6.17E-10 & 3.94E-12 & 4.23E-09 & -8.26E-10
		&-4.66E-09 &3.62E-09 & \textbf{Insur Modest}                  & 1.01E-05                         & 1.40E-06                         & 3.31E-05                         & 1.01E-05                         & 1.40E-06                         & 3.31E-05                         \\
		\textbf{Wine Severe} &1.32E-10 & 3.39E-12 & 1.84E-09& -8.30E-11 &-4.07E-10 &1.80E-09 & \textbf{Insur Severe}         & 2.57E-06                         & 4.90E-07                         & 9.56E-06                         & 2.57E-06                         & 4.90E-07                         & 9.56E-06\\
		\textbf{Build Modest} & 1.49E-07 & 1.93E-09 & 5.91E-07 &-2.19E-05 &-3.37E-05 &-1.28E-05 &\textbf{Blog Modest}              & 8.07E-09                         & 3.33E-10                         & 3.82E-08                         & 8.07E-09                         & 3.33E-10                         & 3.82E-08\\
		\textbf{Build Severe} &3.02E-08 & 4.02E-10 & 1.25E-07 &-1.96E-06 &-3.06E-06 &-1.14E-06 & \textbf{Blog Severe}              & 3.55E-08                         & 2.80E-10                         & 2.24E-07                         & 3.55E-08                         & 2.80E-10                         & 2.24E-07  \\		
		\bottomrule
			\label{tab:SPGLS-real-reltol}
	\end{tabular}
}
%\end{table}%
%\begin{table}[!h]
		%\vspace{-0.2in}
	\setlength\tabcolsep{3pt}
	\centering
	\caption{Relative error of MSEs on test sets}
\resizebox{\linewidth}{!}{
	\begin{tabular}{@{}lcccccclcccccc@{}}
		\toprule
		%		\multicolumn{13}{c}{$m=2n$}\\
		%		\midrule
		\multicolumn{1}{c}{\multirow{2}{*}{\bf{{Dataset}}}}
		&\multicolumn{3}{c}{ $\bf{(M_\textbf{SOCP}-M_\textbf{LTRSR})/|M_\textbf{SOCP}|}$}
		&\multicolumn{3}{c}{ $\bf{(M_\textbf{SOCP}-M_\textbf{RTRNew})/|M_\textbf{SOCP}|}$}
		&\multicolumn{1}{c}{\multirow{2}{*}{\bf{{Dataset}}}}
		&\multicolumn{3}{c}{ $\bf{(M_\textbf{SOCP}-M_\textbf{LTRSR})/|M_\textbf{SOCP}|}$}
		&\multicolumn{3}{c}{ $\bf{(M_\textbf{SOCP}-M_\textbf{RTRNew})/|M_\textbf{SOCP}|}$}
		\\
		\cmidrule(r){2-4} \cmidrule(r){5-7} \cmidrule(r){9-11} \cmidrule(r){12-14}
		%		&\textbf{AVG} & \textbf{MIN} & \textbf{MAX} & \textbf{AVG} & \textbf{MIN} & \textbf{MAX} \\
		& \multicolumn{1}{c}{\textbf{AVG}} & \multicolumn{1}{c}{\textbf{MIN}} & \multicolumn{1}{c}{\textbf{MAX}} & \multicolumn{1}{c}{\textbf{AVG}} & \multicolumn{1}{c}{\textbf{MIN}} & \multicolumn{1}{c}{\textbf{MAX}} && \multicolumn{1}{c}{\textbf{AVG}} & \multicolumn{1}{c}{\textbf{MIN}} & \multicolumn{1}{c}{\textbf{MAX}} & \multicolumn{1}{c}{\textbf{AVG}} & \multicolumn{1}{c}{\textbf{MIN}} & \multicolumn{1}{c}{\textbf{MAX}} \\ \midrule

        \textbf{Wine Modest}  & 1.99E-07  & -9.55E-07 & 4.31E-06 & 1.74E-07  & -9.54E-07 & 3.99E-06 & \textbf{Insur Modest} & 1.05E-05  & 7.69E-07  & 4.21E-05 & 1.02E-05  & -7.09E-07 & 4.28E-05 \\
        \textbf{Wine Severe}  & 2.03E-07  & -1.79E-07 & 2.57E-06 & 2.64E-07  & -4.55E-07 & 2.79E-06 & \textbf{Insur Severe} & 2.43E-06  & -3.36E-07 & 1.01E-05 & 2.45E-06  & -3.27E-07 & 1.01E-05 \\
        \textbf{Build Modest} & -4.52E-06 & -1.25E-04 & 1.23E-04 & -1.25E-04 & -5.27E-04 & 5.01E-04 & \textbf{Blog Modest}  & -2.59E-09 & -5.66E-07 & 2.99E-07 & 1.92E-09  & -5.62E-07 & 3.03E-07 \\
        \textbf{Build Severe} & 2.65E-07  & -2.63E-05 & 2.54E-05 & 8.27E-06  & -8.02E-05 & 1.47E-04 & \textbf{Blog Severe}  & -2.06E-08 & -8.12E-07 & 3.24E-07 & -5.02E-09 & -7.81E-07 & 3.06E-07\\

		\bottomrule
			\label{tab:SPGLS-real-MSE}
	\end{tabular}
}
\end{table}%
%\vspace{-0.2in}
\subsection{Synthetic Dataset}
\subsubsection{Dense Data}\label{ap:reldifrealdata}
Table \ref{tab:SPGLS-dense-reltol} summarises relative errors of objective values on synthetic datasets without sparsity. Comparing to the result in real-world dataset, we find that all the methods
have  high accuracy as the relative errors of the
MSEs are up to \texttt{4.60e-5}.
\begin{table}[H]
%\vspace{-0.2in}
\setlength\tabcolsep{3pt}
\centering
\caption{Relative error on synthetic dataset without sparsity}
\resizebox{\linewidth}{!}{
\begin{tabular}{@{}lcccccclcccccc@{}}
	\toprule
	%		\multicolumn{13}{c}{$m=2n$}\\
	%		\midrule
	\multirow{2}{*}{${\gamma = 0.1}$}
	&\multicolumn{3}{c}{ $\bf{(f_\textbf{SOCP}-f_\textbf{LTRSR})/|f_\textbf{SOCP}|}$}
	&\multicolumn{3}{c}{ $\bf{(f_\textbf{SOCP}-f_\textbf{RTRNew})/|f_\textbf{SOCP}|}$}
	&\multicolumn{1}{c}{\multirow{2}{*}{${\gamma = 0.01}$}}
	&\multicolumn{3}{c}{ $\bf{(f_\textbf{SOCP}-f_\textbf{LTRSR})/|f_\textbf{SOCP}|}$}
	&\multicolumn{3}{c}{ $\bf{(f_\textbf{SOCP}-f_\textbf{RTRNew})/|f_\textbf{SOCP}|}$}
	\\
	\cmidrule(r){2-4} \cmidrule(r){5-7} \cmidrule(r){9-11} \cmidrule(r){12-14}
	%		&\textbf{AVG} & \textbf{MIN} & \textbf{MAX} & \textbf{AVG} & \textbf{MIN} & \textbf{MAX} \\
	& \multicolumn{1}{c}{\textbf{AVG}} & \multicolumn{1}{c}{\textbf{MIN}} & \multicolumn{1}{c}{\textbf{MAX}} & \multicolumn{1}{c}{\textbf{AVG}} & \multicolumn{1}{c}{\textbf{MIN}} & \multicolumn{1}{c}{\textbf{MAX}} && \multicolumn{1}{c}{\textbf{AVG}} & \multicolumn{1}{c}{\textbf{MIN}} & \multicolumn{1}{c}{\textbf{MAX}} & \multicolumn{1}{c}{\textbf{AVG}} & \multicolumn{1}{c}{\textbf{MIN}} & \multicolumn{1}{c}{\textbf{MAX}} \\ \midrule
	$m=2n$ &1.01E-09 &	9.58E-11 &	3.41E-09 &	1.01E-09 &	9.58E-11 &	3.41E-09 & $m=2n   $                  & 1.51E-05                         & 5.39E-08                         & 4.60E-05                         & 1.51E-05                         & 5.39E-08                         & 4.60E-05                         \\
	$m=n$ & 2.26E-08 &	3.37E-11  &	1.34E-07	& 2.26E-08 &3.36E-11 &	1.34E-07 & $m=n$                      & 2.11E-06                         & 3.26E-07                         & 7.40E-06                         & 2.11E-06                         & 3.26E-07                         & 7.40E-06
	\\
	$m=0.5n$ & 8.27E-09 &	2.54E-12 &	4.82E-08 &	8.27E-09 &	2.05E-12 &	4.82E-08 & $m=0.5n $                  & 2.21E-06                         & 3.44E-07                         & 6.87E-06                         & 2.21E-06                         & 3.44E-07                         & 6.87E-06\\
	\bottomrule
		\label{tab:SPGLS-dense-reltol}
\end{tabular}
}
%\vspace{-0.2in}
\end{table}%
\subsubsection{Sparse Data}\label{ap:reldiffsyn}
Table \ref{tab:SPGLS-sparse-reltol} summarises relative error of objective values  in synthetic dataset with different sparsity. Similar to the previous cases, both of our methods have high accuracy in terms of MSEs. These consistent results further prove the validity of our SCLS reformulation.
\begin{table}[H]

%\vspace{-0.1in}

\setlength\tabcolsep{3pt}
\centering
\caption{Relative error on synthetic dataset without sparsity}
\resizebox{\linewidth}{!}{ 	
\begin{tabular}{@{}lcccccclcccccc@{}}
	\toprule
	%		\multicolumn{13}{c}{$m=2n$}\\
	%		\midrule
	\multicolumn{14}{c}{sparsity = 0.0001}\\
	\midrule
	\multirow{2}{*}{${\gamma = 0.1}$}
	&\multicolumn{3}{c}{ $\bf{(f_\textbf{SOCP}-f_\textbf{LTRSR})/|f_\textbf{SOCP}|}$}
	&\multicolumn{3}{c}{ $\bf{(f_\textbf{SOCP}-f_\textbf{RTRNew})/|f_\textbf{SOCP}|}$}
	&\multicolumn{1}{c}{\multirow{2}{*}{${\gamma = 0.01}$}}
	&\multicolumn{3}{c}{ $\bf{(f_\textbf{SOCP}-f_\textbf{LTRSR})/|f_\textbf{SOCP}|}$}
	&\multicolumn{3}{c}{ $\bf{(f_\textbf{SOCP}-f_\textbf{RTRNew})/|f_\textbf{SOCP}|}$}
	\\
	\cmidrule(r){2-4} \cmidrule(r){5-7} \cmidrule(r){9-11} \cmidrule(r){12-14}
	%		&\textbf{AVG} & \textbf{MIN} & \textbf{MAX} & \textbf{AVG} & \textbf{MIN} & \textbf{MAX} \\
	& \multicolumn{1}{c}{\textbf{AVG}} & \multicolumn{1}{c}{\textbf{MIN}} & \multicolumn{1}{c}{\textbf{MAX}} & \multicolumn{1}{c}{\textbf{AVG}} & \multicolumn{1}{c}{\textbf{MIN}} & \multicolumn{1}{c}{\textbf{MAX}} && \multicolumn{1}{c}{\textbf{AVG}} & \multicolumn{1}{c}{\textbf{MIN}} & \multicolumn{1}{c}{\textbf{MAX}} & \multicolumn{1}{c}{\textbf{AVG}} & \multicolumn{1}{c}{\textbf{MIN}} & \multicolumn{1}{c}{\textbf{MAX}} \\
    \midrule
	$m =3n$  & 1.24E-09                         & 1.48E-11                         & 5.58E-09                         & 1.24E-09                         & 1.45E-11                         & 5.58E-09                          & $m =3n $ & 3.98E-09 & 5.27E-11 & 1.19E-08 & 3.97E-09 & 3.19E-11  & 1.19E-08 \\
	$m=2n $ & 7.49E-11                         & 3.53E-13                         & 1.55E-10                         & 7.47E-11                         & 3.20E-13                         & 1.54E-10                          &$m=2n $  & 2.25E-09 & 2.62E-12 & 4.88E-09 & 2.24E-09 & -3.51E-12 & 4.87E-09 \\
	$m=n $  & 9.10E-10                         & 9.93E-12                         & 2.01E-09                         & 9.10E-10                         & 9.73E-12                         & 2.01E-09                          &$m=n $   & 7.95E-09 & 1.83E-10 & 2.34E-08 & 7.94E-09 & 1.16E-10  & 2.34E-08 \\
	$m=0.5n$ & 3.06E-10                         & 3.81E-12                         & 6.60E-10                         & 3.06E-10                         & 3.70E-12                         & 6.60E-10                          & $m=0.5n$ & 3.07E-09 & 7.96E-13 & 6.28E-09 & 3.06E-09 & -4.52E-12 & 6.28E-09\\
	\midrule
	\multicolumn{14}{c}{sparsity = 0.001}\\
	\midrule
	\multirow{2}{*}{${\gamma = 0.1}$}
	&\multicolumn{3}{c}{ $\bf{(f_\textbf{SOCP}-f_\textbf{LTRSR})/|f_\textbf{SOCP}|}$}
	&\multicolumn{3}{c}{ $\bf{(f_\textbf{SOCP}-f_\textbf{RTRNew})/|f_\textbf{SOCP}|}$}
	&\multicolumn{1}{c}{\multirow{2}{*}{$\gamma = 0.01$}}
	&\multicolumn{3}{c}{ $\bf{(f_\textbf{SOCP}-f_\textbf{LTRSR})/|f_\textbf{SOCP}|}$}
	&\multicolumn{3}{c}{ $\bf{(f_\textbf{SOCP}-f_\textbf{RTRNew})/|f_\textbf{SOCP}|}$}
	\\
	\cmidrule(r){2-4} \cmidrule(r){5-7} \cmidrule(r){9-11} \cmidrule(r){12-14}
	%		&\textbf{AVG} & \textbf{MIN} & \textbf{MAX} & \textbf{AVG} & \textbf{MIN} & \textbf{MAX} \\
	& \multicolumn{1}{c}{\textbf{AVG}} & \multicolumn{1}{c}{\textbf{MIN}} & \multicolumn{1}{c}{\textbf{MAX}} & \multicolumn{1}{c}{\textbf{AVG}} & \multicolumn{1}{c}{\textbf{MIN}} & \multicolumn{1}{c}{\textbf{MAX}} && \multicolumn{1}{c}{\textbf{AVG}} & \multicolumn{1}{c}{\textbf{MIN}} & \multicolumn{1}{c}{\textbf{MAX}} & \multicolumn{1}{c}{\textbf{AVG}} & \multicolumn{1}{c}{\textbf{MIN}} & \multicolumn{1}{c}{\textbf{MAX}} \\
    \midrule
  $m=3n $ & 1.36E-09                         & 5.73E-11                         & 2.58E-09                         & 1.36E-09                         & 5.64E-11                         & 2.58E-09                          &$m =3n$  & 1.92E-08 & 5.04E-10 & 8.89E-08 & 1.92E-08 & 5.04E-10 & 8.89E-08 \\
   $m=2n $  & 1.66E-09                         & 6.00E-11                         & 5.06E-09                         & 1.66E-09                         & 5.91E-11                         & 5.06E-09                          & $m=2n $  & 5.14E-08 & 9.80E-12 & 2.05E-07 & 5.14E-08 & 9.74E-12 & 2.05E-07 \\
  $ m=n  $  & 5.76E-10                         & 1.53E-09                         & 1.13E-12                         & 6.19E-09                         & 1.53E-09                         & -2.03E-13                         & $m=n  $  & 1.96E-08 & 1.14E-10 & 4.83E-08 & 1.96E-08 & 1.14E-10 & 4.83E-08 \\
   $ m=0.5n$ & 1.90E-09                         & 1.02E-10                         & 6.68E-09                         & 1.90E-09                         & 1.01E-10                         & 6.68E-09                          & $m=0.5n $& 4.42E-08 & 4.91E-09 & 1.34E-07 & 4.42E-08 & 4.91E-09 & 1.34E-07  \\  	
	\midrule
	\multicolumn{14}{c}{sparsity = 0.01}\\
	\midrule
	\multirow{2}{*}{${\gamma = 0.1}$}
	&\multicolumn{3}{c}{ $\bf{(f_\textbf{SOCP}-f_\textbf{LTRSR})/|f_\textbf{SOCP}|}$}
	&\multicolumn{3}{c}{ $\bf{(f_\textbf{SOCP}-f_\textbf{RTRNew})/|f_\textbf{SOCP}|}$}
	&\multicolumn{1}{c}{  \multirow{2}{*}{${\gamma = 0.01}$}}
	&\multicolumn{3}{c}{ $\bf{(f_\textbf{SOCP}-f_\textbf{LTRSR})/|f_\textbf{SOCP}|}$}
	&\multicolumn{3}{c}{ $\bf{(f_\textbf{SOCP}-f_\textbf{RTRNew})/|f_\textbf{SOCP}|}$}
	\\
	\cmidrule(r){2-4} \cmidrule(r){5-7} \cmidrule(r){9-11} \cmidrule(r){12-14}
	%		&\textbf{AVG} & \textbf{MIN} & \textbf{MAX} & \textbf{AVG} & \textbf{MIN} & \textbf{MAX} \\
	& \multicolumn{1}{c}{\textbf{AVG}} & \multicolumn{1}{c}{\textbf{MIN}} & \multicolumn{1}{c}{\textbf{MAX}} & \multicolumn{1}{c}{\textbf{AVG}} & \multicolumn{1}{c}{\textbf{MIN}} & \multicolumn{1}{c}{\textbf{MAX}} && \multicolumn{1}{c}{\textbf{AVG}} & \multicolumn{1}{c}{\textbf{MIN}} & \multicolumn{1}{c}{\textbf{MAX}} & \multicolumn{1}{c}{\textbf{AVG}} & \multicolumn{1}{c}{\textbf{MIN}} & \multicolumn{1}{c}{\textbf{MAX}} \\
    \midrule
        $m =3n$  & 2.19E-08                         & 2.64E-11                         & 1.70E-08                         & 2.19E-08                         & 2.63E-11                         & 1.70E-08                          & $m =3n$  & 1.05E-07 & 2.15E-09 & 2.51E-07 & 1.05E-07 & 2.15E-09 & 2.51E-07 \\
    $m=2n$   & 8.23E-08                         & 8.23E-08                         & 3.74E-07                         & 1.02E-07                         & 4.80E-11                         & 3.74E-07                          & $m=2n$  & 6.94E-07 & 3.99E-10 & 3.10E-06 & 6.94E-07 & 3.99E-10 & 3.10E-06 \\
    $m=n$    & 1.27E-08                         & 9.14E-10                         & 3.37E-08                         & 1.27E-08                         & 9.14E-10                         & 3.37E-08                          &$m=n $   & 2.21E-06 & 7.78E-09 & 1.00E-05 & 2.21E-06 & 7.78E-09 & 1.00E-05 \\
   $ m=0.5n $& 1.90E-08                         & 3.52E-12                         & 4.50E-08                         & 1.90E-08                         & 3.45E-12                         & 4.50E-08                          & $m=0.5n$ & 6.52E-06 & 6.40E-08 & 2.95E-05 & 6.52E-06 & 6.40E-08 & 2.95E-05\\
	\bottomrule
		\label{tab:SPGLS-sparse-reltol}
\end{tabular}
}
\end{table}%

\end{document}